\documentclass[11pt]{article}

\usepackage{fullpage}
\usepackage{parskip}
\usepackage{tikz} 
\usepackage{hyperref}
\usepackage{eqnarray,amsthm, amssymb, amsmath,verbatim}
\usepackage{mathrsfs}
\usepackage[margin=1in]{geometry}
\usepackage{enumerate}
\usepackage{graphicx}
\usetikzlibrary{calc,shadings}
\usepackage{pgfplots}
\usepackage{slashbox}
\usetikzlibrary{patterns,arrows,decorations.pathreplacing}
\newcolumntype{C}[1]{>{\centering\let\newline\\\arraybackslash\hspace{0pt}}m{#1}}

\newtheorem{theorem}{Theorem}
\newtheorem{lemma}[theorem]{Lemma}
\newtheorem{corollary}[theorem]{Corollary}
\newtheorem{definition}{Definition}
\newtheorem{remark}{Remark}
\newtheorem{conjecture}[theorem]{Conjecture}
\newtheorem{result}[theorem]{Lemma}

\newcommand{\tref}[1]{Theorem \ref{theorem:#1}}
\newcommand{\lref}[1]{Lemma \ref{lemma:#1}}
\newcommand{\fref}[1]{Figure \ref{fig:#1}}
\newcommand{\taref}[1]{Table \ref{table:#1}}
\newcommand{\cref}[1]{Conjecture \ref{conjecture:#1}}

\def\addlegendimage{\csname pgfplots@addlegendimage\endcsname}
\pgfkeys{/pgfplots/number in legend/.style={%
		/pgfplots/legend image code/.code={%
			\node at (0.15,-0.0225){#1};
		},%
	},
}
\pgfplotsset{
	every legend to name picture/.style={west}
}

\newcommand{\PFmnu}[4]{
	\coordinate (axis) at (#1);
	\coordinate (axis2) at (0,0);
	\draw[help lines] (#1) grid +(#2,#3);
	\draw[dashed] (#1) -- +(#2,#3);
	\coordinate (prev) at (#1);
	
	\foreach \x/\y in {#4}{	
		\ifnum\y=0
		\draw[line width=2pt,blue] (axis)+(axis2) -- +(#2,0);
		\else
		\draw[line width=2pt,blue] (axis)+(\x,0) -- +(\x,1);
		\draw[line width=2pt,blue] (axis)+(axis2) -- +(\x,0);
		\path (axis)+(\x+0.5,0.5) node {$\y$};		
		\fi
		
		\path (axis) -- +(0,1) coordinate (axis);
		
		\path (axis2) -- (\x,0) coordinate (axis2);
	}
}
\newcommand{\PFmnum}[4]{
	\coordinate (axis) at (#1);
	\coordinate (axis2) at (0,0);
	\draw[help lines] (#1) grid +(#2,#3);
	\draw[dashed] (#1) -- +(#2,#3);
	\coordinate (prev) at (#1);
	\foreach \x/\y in {#4}{	
		\ifnum\y=0
		\draw[line width=1pt,blue] (axis)+(axis2) -- +(#2,0);
		\else
		\draw[line width=1pt,blue] (axis)+(\x,0) -- +(\x,1);
		\draw[line width=1pt,blue] (axis)+(axis2) -- +(\x,0);
		\path (axis)+(\x+0.5,0.5) node {\tiny$\y$};		
		\fi
		\path (axis) -- +(0,1) coordinate (axis);
		\path (axis2) -- (\x,0) coordinate (axis2);
	}
}

\newcommand{\Dpath}[4]{
	\coordinate (axis) at (#1);
	\coordinate (axis2) at (#1);
	\draw[help lines] (#1) grid +(#2,#3);
	\draw[dashed] (#1) -- +(#2,#3);
	\foreach \x in {#4}{	
		\ifnum\x=-1
		\draw[line width=2pt,blue] (axis)+(axis2) -- +(#2,0);
		\else
		\draw[line width=2pt,blue] (axis)+(\x,0) -- +(\x,1);
		\draw[line width=2pt,blue] (axis)+(axis2) -- +(\x,0);
		\path (axis) -- +(0,1) coordinate (axis);
		\path (axis2) -- (\x,0) coordinate (axis2);	
		\fi
	}
}

\newcommand{\fillshade}[1]{\foreach \x/\y in {#1}{\path[fill,blue!20!white] (\x-1,\y-1) rectangle (\x,\y);}}
\newcommand{\fillshadesmall}[1]{\foreach \x/\y in {#1}{\path[fill,blue!20!white] (\x-.9,\y-.9) rectangle (\x-.1,\y-.1);}}

\newcommand{\fillshadea}[1]{\foreach \x/\y in {#1}{\path[fill,orange!50!yellow!70!white] (\x-1,\y-1) rectangle (\x,\y);}}

\newcommand{\PFtext}[2]{
	\coordinate (axis) at (#1);
	\coordinate (axis2) at (#1);
	\foreach \x/\y in {#2}{	
		\path (axis)+(\x+0.5,0.5) node {$\y$};		
		\path (axis) -- +(0,1) coordinate (axis);
		\path (axis2) -- (\x,0) coordinate (axis2);
}}
\newcommand{\PFtextr}[4]{
	\coordinate (axis) at (0,0);
	\coordinate (axis2) at (0,0);
	\draw[help lines] (#1) grid +(#2,#3);
	\draw[dashed] (#1) -- +(#2,#3);
	\coordinate (prev) at (#1);
	
	\foreach \x/\y in {#4}{	
		\path (axis)+(\x+0.25,0.25) node {\tiny$\y$};		
		\path (axis) -- +(0,1) coordinate (axis);
		\path (axis2) -- (\x,0) coordinate (axis2);
}}

\newcommand{\fillls}[3]{\node at (#1,#2) {\footnotesize$#3$};}
\newcommand{\fille}[3]{\node[anchor=east] at (#1,#2) {\footnotesize$#3$};}
\newcommand{\fillw}[3]{\node[anchor=west] at (#1,#2) {\footnotesize$#3$};}
\newcommand{\fillll}[3]{\node at (#1-.5,#2-.5) {\footnotesize$#3$};}
\newcommand{\fillsome}[1]{\foreach \x/\y/\z in {#1}{\node at (\x-.5,\y-.5) {\footnotesize$\z$};}}
\newcommand\fillcir[1]{\foreach \x/\y in {#1}{\node at (\x-.5,\y-.5) {$\bigcirc$};}}
\newcommand\fillcro[1]{\foreach \x/\y in {#1}{\node at (\x-.5,\y-.5) {$\times$};}}
\newcommand\filltri[1]{\foreach \x/\y in {#1}{\node at (\x-.5,\y-.5) {$\bigtriangleup$};}}
\newcommand\fillpoint[1]{\foreach \x/\y in {#1}{\draw[red, thick,fill] (\x,\y) circle (2mm);}}

\newcommand{\Dyck}{\mathcal{D}}
\newcommand{\PF}{\pi}
\newcommand{\PFc}{\mathcal{P}}
\newcommand{\dinv}{\mathrm{dinv}}
\newcommand{\tdinv}{\mathrm{tdinv}}
\newcommand{\maxdinv}{\mathrm{maxdinv}}
\newcommand{\pdinv}{\mathrm{pdinv}}
\newcommand{\dinvcorr}{\mathrm{dinvcorr}}
\newcommand{\area}{\mathrm{area}}
\newcommand{\rank}{\mathrm{rank}}
\newcommand{\ides}{\mathrm{ides}}
\newcommand{\pides}{\mathrm{pides}}
\newcommand{\word}{\mathrm{word}}

\newcommand{\arm}{\mathrm{arm}}
\newcommand{\leg}{\mathrm{leg}}
\newcommand{\Hikita}{\mathrm{H}}
\newcommand{\Qmn}[1]{\mathrm{Q}_{#1}}
\newcommand{\qtn}[1]{[#1]_{q,t}}
\newcommand{\qtnn}[1]{[#1]_{q^2,t^2}}

\newcommand{\sg}{\sigma}

\newcommand{\Sn}[1]{\mathcal{S}_{#1}}
\newcommand{\Des}[1]{\mathrm{Des}}
\newcommand{\swi}{\mathbb{S}}

\newcommand{\partitiontwo}[6]{
	\draw (#1,#2) rectangle +(#3,-1);
	\node at (#1+#3/2,#2-.5) {\footnotesize
		$#4$};
	\draw (#1,#2)+(0,-1) rectangle +(#5,-2);
	\node at (#1+#5/2,#2-1.5) {\footnotesize
		$#6$};
}
\newcommand{\arrowx}[2]{
	\draw[-latex,very thick] (#1)--(#2);
}
\newcommand{\dotdot}[3]{
	\draw[fill] (#1) circle [radius=0.05];
	\draw[fill] (#2) circle [radius=0.05];
	\draw[fill] (#3) circle [radius=0.05];
}

\newcommand{\scoeff}[1]{[s_{#1}]}
\newcommand{\sqt}[1]{s_{#1}(q,t)}

\newcommand{\ret}{\mathrm{ret}}
\newcommand{\al}{\alpha}
\newcommand{\PFaa}{\mathcal{P}_{3,n | \lambda 13 \mu}}
\newcommand{\PFbb}{\mathcal{P}_{3,n | \lambda 22 \mu}}
\newcommand{\la}{\lambda}
\newcommand{\hstr}{\mathit{hstr}}
\newcommand{\vstr}{\mathit{vstr}}

\title{Schur Function Expansions and the Rational Shuffle Theorem}

\author{
	Dun Qiu\\[-0.8ex]
	\small Department of Mathematics\\[-0.8ex]
	\small Beijing Jiaotong University\\[-0.8ex]
	\small Beijing, P. R. China\\[-0.8ex]
	\small \texttt{qiudun123@163.com}
	\and 
	Jeffrey Remmel \\
	\small Department of Mathematics\\[-0.8ex]
	\small University of California, San Diego\\[-0.8ex]
	\small La Jolla, CA, USA\\[-0.8ex]
}
\date{}

\begin{document}

\maketitle

\begin{abstract}
	\noindent Gorsky and Negut introduced operators $\mathrm{Q}_{m,n}$ on symmetric functions and conjectured that, in the case where $m$ and $n$ are relatively prime, the expression $\mathrm{Q}_{m,n}(1)$ is given by the Hikita polynomial $\mathrm{H}_{m,n}[X;q,t]$. 
	Later, Bergeron-Garsia-Leven-Xin
	extended and refined the conjectures of  $\mathrm{Q}_{m,n}(1)$ for arbitrary $m$ and $n$ which 
	we call the Extended Rational Shuffle Conjecture. 
	In the special case $\mathrm{Q}_{n+1,n}(1)$, the Rational Shuffle Conjecture becomes the Shuffle Conjecture of Haglund-Haiman-Loehr-Remmel-Ulyanov, which was proved in 2015 by Carlsson and Mellit as the Shuffle Theorem. The Extended Rational Shuffle Conjecture was later proved by Mellit as the Extended Rational Shuffle Theorem.  The main goal of this paper is 
	to study 
	the combinatorics of the coefficients that arise in the Schur function expansion 
	of $\mathrm{Q}_{m,n}(1)$ in certain special cases. 
	Leven gave a combinatorial proof of the Schur function expansion of 
	$\mathrm{Q}_{2,2n+1}(1)$ and $\mathrm{Q}_{2n+1,2}(1)$. 
	In this paper, we explore several symmetries in the combinatorics of the coefficients that arise in the Schur function expansion of $\mathrm{Q}_{m,n}(1)$. Especially, we study the hook-shaped Schur function coefficients, and the Schur function expansion of $\mathrm{Q}_{m,n}(1)$ in the case where $m$ or $n$ equals $3$.\\
	
\end{abstract}

\section{Introduction}

The Rational Shuffle Theorem, as a rational generalization of the Shuffle Theorem, comes from the study of the ring of diagonal harmonics. Let $X=\{x_1,x_2,\ldots,x_n\}$ and $Y=\{y_1,y_2,\ldots,y_n\}$ be two sets of $n$ variables. The {\em ring of diagonal harmonics} consists of 
those polynomials in $\mathbb{Q}[X,Y]$ which satisfy 
the following system of differential equations
$$
\partial_{x_1}^a\partial_{y_1}^b\,f(X,Y)+\partial_{x_2}^a\partial_{y_2}^b\,f(X,Y)+\ldots+\partial_{x_n}^a\partial_{y_n}^b\,f(X,Y)=0
$$
for each pair of integers $a$ and $b$ such that {$a+b>0$}. Haiman in \cite{Haiman} proved that the ring of diagonal harmonics has {dimension} $(n+1)^{n-1}$, and the {\em bigraded Frobenius characteristic} of the $\Sn{n}$-module of diagonal harmonics, $DH_n(X;q,t)$, is given by
\begin{equation}
DH_n(X;q,t)=\nabla e_n,
\end{equation}
where $\nabla$ ({\em nabla}) is the symmetric function operator defined by Bergeron and Garsia \cite{GB}, and $e_n$ is the {\em elementary symmetric function} of degree $n$.

Let $n$ be a positive integer. An $(n,n)$-{\em Dyck path} $P$ is a lattice path from $(0,0)$ to $(n,n)$ which always remains weakly above the main diagonal $y=x$. 
Given a Dyck path $P$, we can get an {\em $(n,n)$-parking function} $\pi$ by labeling the cells east of and adjacent to the north steps of $P$ with integers $\{1,\ldots,n\}$ such that the labels are strictly increasing in each column. The set of parking functions of size $n$ is denoted by $\PFc_n$.  
\fref{PF} (a) gives an example of a $(5,5)$-parking function. 

\begin{figure}[ht!]
	\centering
	\begin{tikzpicture}[scale=0.6]
	\Dpath{0,0}{5}{5}{0,0,1,1,3,-1};
	\PFtext{0,0}{0/1,0/3,1/4,1/5,3/2};
	\node at (2.5,-.7) {(a)};
	\end{tikzpicture}
	\begin{tikzpicture}[scale=0.6]
	\PFmnu{0,0}{3}{5}{0/2,0/4,1/1,1/3,1/5,0/0};
	\node at (1.5,-.7) {(b)};
	\path (-4,0);
	\end{tikzpicture}
	\vspace*{-3mm}
	\caption{A $(5,5)$-parking function and a $(3,5)$-parking function.}
	\label{fig:PF}
\end{figure}
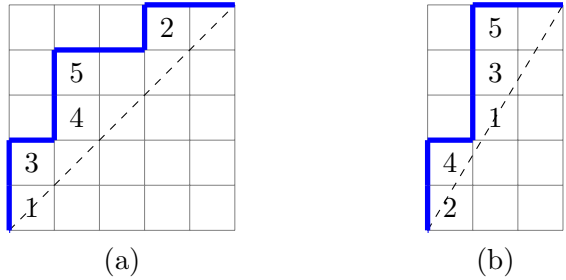

Let $F_{\alpha}[X]$ denote the fundamental quasi-symmetric function of Gessel \cite{gessel} associated to the composition $\alpha$, and let $\area$, $\dinv$, $\pides$ be statistics of parking functions.
The {\em Classical Shuffle Conjecture} proposed by Haglund, Haiman, Loehr, Remmel and Ulyanov \cite{HHLRU} gives a well-studied combinatorial expression for the bigraded Frobenius characteristic of the ring of diagonal harmonics. The Shuffle Conjecture has been proved by Carlsson and Mellit \cite{CM} as  the \emph{Shuffle Theorem} that for all $n\geq 0$, 
\begin{equation}
\nabla e_n=\sum_{\PF\in\PFc_n}t^{\area(\PF)}q^{\dinv(\PF)}F_{\pides(\PF)}[X].
\end{equation}

Let $m$ and $n$ be positive integers.  An \emph{$(m, n)$-Dyck path} is a lattice path from $(0,0)$ to $(m,n)$ which always remains weakly above the main diagonal $y = \frac{n}{m} x$. An $(m, n)$-\textit{parking function} $\PF$ is obtained by labeling the north steps of an $(m, n)$-Dyck path in a similar way to the $(n,n)$ case. 
\fref{PF} (b) gives an example of a $(3,5)$-parking function.
Let $\area$, $\dinv$, $\pides$ and $\ret$ be statistics of rational parking functions (will be defined later), then Bergeron, Garsia, Leven and Xin \cite{BGLX}
extended the combinatorial side of the Shuffle Theorem to the \emph{extended Hikita polynomial}
\begin{equation}
\Hikita_{m,n}[X; q, t]:=\sum_{\PF\in\PFc_{m,n}} [ret(\PF)]_{\frac{1}{t}}t^{\area(\PF)}q^{\dinv(\PF)}F_{\pides(\PF)}[X].
\end{equation}
Gorsky and Negut  introduced the symmetric function operator $\Qmn{m,n}$ for $m$ and $n$ coprime in \cite{GN} (where they call the operator $\widetilde{\mathcal{P}}_{m,n}$), and Bergeron,
Garsia, Leven and Xin \cite{BGLX} generalized the operator $\Qmn{m,n}$ so that the coprimality condition is removed, extending the algebraic side of the Shuffle Theorem from $\nabla e_n$ to $\Qmn{m,n}(1)$. 
The {\em Extended Rational Shuffle Theorem}  is that, 
for any pair of positive integers $(m,n)$, 
\begin{equation}\label{rationalshuffle}
\Qmn{m,n}(1)=\Hikita_{m,n}[X;q,t],
\end{equation}
which was proved by Mellit \cite{Mellit}.

A more important goal is to find the Schur function expansion of $\nabla e_n$ since that would allow us to find the bigraded $\Sn{n}$-isomorphism type
of the ring of diagonal harmonics, see \cite{Haiman}. More generally, we would like to find a combinatorial interpretation of the coefficients that arise in the Schur function expansion of $\Qmn{m,n}(1)$. The main goal of this paper is to find such Schur function expansions in the case where $m$ or $n$ equals 3. The Schur function expansion of $\Qmn{m,n}(1)$ in the case where $m$ and $n$ are coprime and either $m$ or $n$ equals 2 was given by Leven \cite{Em}. That is, let 
$\qtn{n}$ be  the $q,t$-analogue of the integer $n$ that 
\begin{equation}
\qtn{n}:=\frac{q^n-t^n}{q-t}=q^{n-1}+q^{n-2}t+\cdots+t^{n-1},
\end{equation}
then Leven \cite{Em} gave a proof of the following theorem. 

\begin{theorem}[Leven]
	For any integer $k\geq 0$,
	\begin{equation}\label{emily1}
	\Qmn{2k+1,2}(1)=\Hikita_{2k+1,2}[X; q, t]=\qtn{k}s_2+\qtn{k+1}s_{11},
	\end{equation}
	and
	\begin{equation}\label{emily2}
	\Qmn{2,2k+1}(1)=\Hikita_{2,2k+1}[X; q, t]=\sum_{r=0}^{k}\qtn{k+1-r}s_{2^r1^{2k+1-2r}}.
	\end{equation}
\end{theorem}

We want to write the coefficients of Schur functions in $\Qmn{m,n}(1)$ as symmetric functions in variables $(q,t)$.
A $(q,t)$-Schur function $\sqt{\la}$ is non-zero only if the partition $\la$ has no more than two parts, and 
\begin{equation}
\sqt{(a,b)} = (qt)^b \qtn{a-b+1}.
\end{equation}
Thus, the right hand side of Equations (\ref{emily1}) and (\ref{emily2}) can be written as $\sqt{k-1} s_2+\sqt{k} s_{11}$ and $\sum_{r=0}^{k}\sqt{k-r}s_{2^r1^{2k+1-2r}}$ respectively.

In fact, the coefficients of Schur functions in $\Qmn{m,n}(1)$ as $(q,t)$-Schur functions are \emph{Schur positive} due to the work of Bergeron \cite{B1} that the ring of diagonal harmonics forms both an $S_n$-module and a $GL_2$-module which commute with each other.

By the Extended Rational Shuffle Theorem formulated in \cite{BGLX}, we can extend Leven's theorem to compute the Schur function expansion of $\Qmn{m,n}(1)$ where either $m$ or $n$ is equal to 2,  
but $m$ and $n$ are not coprime. 
We give a proof of the following theorem in Section 2.2. 
\begin{theorem}\label{theorem:2k2}
	For any integer $k> 0$,
	\begin{equation}\label{eq1}
	\Qmn{2k,2}(1)=\Hikita_{2k,2}[X; q, t]=\left(\sqt{k-1}+\sqt{k-2}\right)s_2+\left(\sqt{k}+\sqt{k-1}\right)s_{11},
	\end{equation}
	and
	\begin{equation}\label{eq2}
	\Qmn{2,2k}(1)=\Hikita_{2,2k}[X; q, t]=\sum_{r=0}^{k}\left(\sqt{k-r}+\sqt{k-r-1}\right)s_{2^r1^{2k+1-2r}}.
	\end{equation}
\end{theorem}

The coefficient at $s_{1^n}$ in $\Qmn{m,n}(1)$ is known as the rational $q,t$-Catalan number, computed by Gorsky and Mazin \cite{GM} for the case $n=3$ and studied by Lee, Li and Loehr \cite{LLL} for the case $n=4$. The coefficients at hook-shaped Schur functions were discussed by Armstrong, Loehr and Warrington \cite{ALW}.

In this paper, we explore the combinatorics of the Schur 
function expansion of $\Qmn{m,n}(1)$ in several special cases.

In Section 2, we provide backgrounds of the problem in both combinatorial side (parking function side) and algebraic side (symmetric function side). Then in Section 3, we prove several symmetries of the coefficients of Schur functions in the Extended Rational Shuffle Theorem. Let $[s_\lambda]_{m,n}$ be the coefficient of the Schur function $s_{\lambda}$ in both $\Qmn{m,n}(1)$ and $\Hikita_{m,n}[X; q, t]$, then we can combinatorially prove
\begin{theorem}\label{theorem:3} For all $m,n > 0$ and $\lambda'\vdash (n-am)$,
	\begin{enumerate}[(a)]
		\item $[s_{1^n}]_{m,n}=[s_{n}]_{m+n,n}$,
		\item $[s_{m^a\lambda'}]_{m,n}=[s_{\lambda'}]_{m,n-a m}$,
		\item $[s_{k1^{n-k}}]_{m,n}=[s_{k1^{m-k}}]_{n,m}$.
	\end{enumerate}
\end{theorem}

\noindent Nakagane \cite{Nakagane} obtained a similar result to \tref{3} independently.

In Section 4, we prove the following theorem to give explicit formulas for the Schur function expansion of $\Qmn{m,3}(1)$ from both symmetric function side and combinatorial side. 
\begin{theorem}\label{theorem:1}For any integer $k\geq 0$,
	\begin{multline}\label{maineqn}
	\Qmn{3k+1,3}(1)=\Hikita_{3k+1,3}[X; q, t]=\left(\sum_{i=0}^{k-1}\sqt{(k+2i-1,k-i-1)}\right)s_3\\
	+\left(\sum_{i=0}^{k-1}\left(
	\sqt{(k+2i,k-i-1)}+\sqt{(k+2i+1,k-i-1)}
	\right)\right)s_{21}+\left(\sum_{i=0}^{k}\sqt{(k+2i,k-i)}\right)s_{111},
	\end{multline}
	\begin{multline}\label{maineqn2}
	\Qmn{3k+2,3}(1)=\Hikita_{3k+2,3}[X; q, t]=\left(\sum_{i=0}^{k-1}\sqt{(k+2i,k-i-1)}\right)s_3\\
	+\left(\sum_{i=-1}^{k-1}\left(
	\sqt{(k+2i+1,k-i-1)}+\sqt{(k+2i+2,k-i-1)}\right)
	\right)s_{21}+\left(\sum_{i=0}^{k}\sqt{(k+2i+1,k-i)}\right)s_{111},
	\end{multline}
	\begin{multline}\label{maineqn3}
	\Qmn{3k,3}(1)=\Hikita_{3k,3}[X; q, t]=\left(\sum_{i=0}^{k-1}
	(\sqt{(k+2i-3,k-i-1)}+\sqt{(k+2i-2,k-i-1)}\right.\\
	+\sqt{(k+2i-1,k-i-1)})
	\bigg)s_3
	+\bigg(\sqt{(k+1,k-1)}+2\sqt{(k,k-1)}+\sqt{(k-1,k-1)}\\
	\left.
	+\sum_{i=1}^{k-1}
	\left(\sqt{k+2i-2,k-i-1}+2\sqt{k+2i-1,k-i-1}+2\sqt{k+2i,k-i-1}+\sqt{k+2i+1,k-i-1}\right)
	\right)s_{21}\\
	+\left(\sum_{i=0}^{k}
	(\sqt{k+2i-2,k-i}+\sqt{k+2i-1,k-i}+\sqt{k+2i,k-i})\right)s_{111}.
	\end{multline}	
\end{theorem}
\noindent Note that this independently proves the Shuffle Theorem and the Extended Rational Shuffle Theorem when $n\leq 3$.

In Section 5, we prove several Schur function coefficient formulas and symmetries in $\Qmn{3,n}(1)$ (some of which are consequences of \tref{3}), and conjecture a concise recursive formula for Schur function coefficients $[s_\lambda]_{3,n}$ generally for any $\lambda\vdash n$.
In particular, we study a new symmetry that 
\begin{equation}
[s_{2^a1^b}]_{3,n}=[s_{2^b1^a}]_{3,3(a+b)-n},
\end{equation}
and a combinatorial action on parking functions called the switch map $\swi$.

\section{Background}
To state our results, we shall first introduce details about the Extended Rational Shuffle Theorem. This 
will require a series of definitions. We omit the word ``Extended" as long as it will not cause any ambiguity.
\subsection{Combinatorial side}

Let $m$ and $n$ be positive integers.  The set of $(m, n)$-Dyck paths is denoted by $\Dyck_{m,n}$.
For an $(m, n)$-Dyck path, the cells that are cut through by the main diagonal will be called  \textit{diagonal cells}. \fref{2} (a) gives an example of a $(5,7)$-Dyck path, and \fref{2} (b) gives an example of a $(4,6)$-Dyck path, where the diagonal cells are shaded.

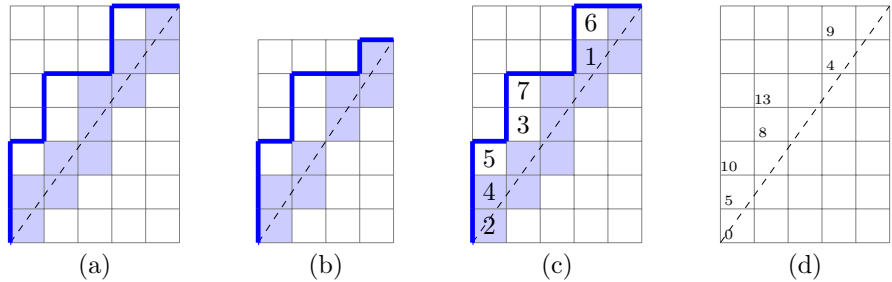
\begin{figure}[ht]
	\centering
	\vspace{-6mm}	
	\scalebox{.9}{\begin{tikzpicture}[scale=0.5]
		\fillshade{1/1,1/2,2/2,2/3,3/3,3/4,3/5,4/5,4/6,5/6,5/7}
		\Dpath{0,0}{5}{7}{0,0,0,1,1,3,3,-1};
		\node at (2.5,-.7) {(a)};
		\path (7,1);
		\end{tikzpicture}
		\begin{tikzpicture}[scale=0.5]
		\fillshade{1/1,1/2,2/2,2/3,3/4,3/5,4/5,4/6}
		\Dpath{0,0}{4}{6}{0,0,0,1,1,3,-1};
		\node at (2,-.7) {(b)};
		\path (6,1);
		\end{tikzpicture}
		\begin{tikzpicture}[scale=0.5]
		\fillshade{1/1,1/2,2/2,2/3,3/3,3/4,3/5,4/5,4/6,5/6,5/7}
		\PFmnu{0,0}{5}{7}{0/2,0/4,0/5,1/3,1/7,3/1,3/6,0/0};
		\node at (2.5,-.7) {(c)};
		\end{tikzpicture}
		\begin{tikzpicture}[scale=0.5]
		\path (-2,0);
		\PFtextr{0,0}{5}{7}{0/0,0/5,0/10,1/8,1/13,3/4,3/9};
		\node at (2.5,-.7) {(d)};
		\end{tikzpicture}}
	\caption{\fontsize{10.5}{0}\selectfont A $(5,7)$-Dyck path, a $(4,6)$-Dyck path, a $(5,7)$-parking function and its car ranks.}
	\label{fig:2}
\end{figure}

For an $(m,n)$-Dyck path, we have the statistic area defined as follows.

\begin{definition}[$\area$]
	The number of full cells between an $(m, n)$-Dyck path $P$ and the main diagonal is denoted by $area(P)$.
\end{definition}

The cells above a Dyck path $P$ are called \emph{coarea cells} of $P$, and they form a Ferrers diagram (in English notation) of a partition $\lambda(P)$. In the example in \fref{2} (a), $\lambda(P)=(3,3,1,1)$ or\ \  \raisebox{-.3\height}{\begin{tikzpicture}[scale=0.2] \draw (3,0) grid (0,2) grid (1,-2);
	\end{tikzpicture}}. 

For any partition $\mu$ and any cell $c$ in the Ferrers diagram of $\mu$, we let $arm(c)$ be the number of cells to the right of $c$ in $\mu$ and $leg(c)$ be the number of cells below $c$ in $\mu$. Let $\chi(x)$ denote the function that takes value $1$ if its argument $x$ is true, and $0$ otherwise, then we can define the path dinv $(\pdinv)$ statistic of an $(m, n)$-Dyck path.

\begin{definition}[pdinv]The $\pdinv$ of an $(m, n)$-Dyck path $P$ is given by 
	$$
	\pdinv(P) := \sum_{c\in\lambda(P)}\chi\left(\frac{arm(c)}{leg(c)+1}\leq\frac{m}{n}<\frac{arm(c)+1}{leg(c)}\right).
	$$
\end{definition}

We can get an $(m, n)$-{parking function} $\PF$ by labeling the north steps of an $(m, n)$-Dyck path
with the integers $\{1, \ldots,n\}$ such that the numbers increase in each column from bottom to top, and we will refer to these labels as \textit{cars}. 
The underlying Dyck path is denoted by $\Pi(\PF)$, and the partition formed by the collection of cells above the path $\Pi(\PF)$ is denoted by $\lambda(\PF)$. The set of $(m, n)$-parking functions is denoted by $\PFc_{m,n}$.
\fref{2} (c)  pictures a $(5,7)$-parking function based on the $(5,7)$-Dyck path in \fref{2} (a).

Next we define statistics $\ides$ and $\pides$ for rational parking functions. 
For any pair of coprime positive integers $m$ and $n$, we define the {\em rank} of a cell $(x,y)$ in the $(m,n)$-grid to be $rank(x,y):=my-nx$. If $m$ and $n$ are not coprime, we shall generalize the rank to be $rank(x,y):=my-nx+\lfloor \frac{x \gcd(m,n)}{m}\rfloor$. \fref{2} (d) shows the ranks of cars in \fref{2} (c). The {\em word} (or {\em diagonal word}), $\sigma(\PF)$ (or $\word(\PF)$),  of $\PF$, is obtained by reading cars from highest to lowest ranks. In our example in \fref{2} (c), $\sigma(\PF)=7563412$. We define 
$\ides(\PF)$ to be the descent set of $\sigma(\PF)^{-1}$. In other words, we have
\begin{definition}[$\ides$] Let $\PF$ be any parking function, then 
\begin{align*}
\ides(\PF)&:=\{i\in\sigma(\PF):i+1 \mbox{ is to the left of } i\mbox{ in }\sigma(\PF)\}\\
&=\{i:\rank(i)<\rank(i+1)\}.
\end{align*}
\end{definition}
Then we define $\pides(\PF)$ to be the composition set of $\ides(\PF)$.
\begin{definition}[$\pides$] For any $(m, n)$-parking function $\PF$, if $\ides(\PF)=\{i_1<i_2<\cdots<i_d\}$, then
$$\pides(\PF):=\{i_1,i_2-i_1,\ldots,n-i_d\}.$$
\end{definition}
In \fref{2} (c), we have $\ides(\PF)=\ides(7563412) = \{2,4,6\}$, and $\pides(\PF)=\{2,2,2,1\}$.

We have the following two remarks about the statistics word, ides and pides. 
\begin{remark}\label{remark:1}
	Let $i<j$ be two cars in the parking function $\PF$. If $i$ is to the left of $j$ in $\sigma(\PF)$, then the cars $i,j$ must be in different columns.
\end{remark}
\begin{proof}
In $\sigma(\PF)$, the allocation of $i$ and $j$ implies that $\rank(i)>\rank(j)$. If $i$ and $j$ are in the same column, then it must be the case that $j$ lies on top of $i$, which leads to a contradiction with $\rank(i)>\rank(j)$. Thus, $i$ and $j$ must be in different columns.
\end{proof}

\begin{remark}\label{remark:2}
	For $\PF\in\PFc_{m,n}$, the parts in the composition set $\pides(\PF)$  are less than or equal to $m$.
\end{remark}
\begin{proof}
Suppose to the contrary. If $M\in\pides(\PF)$ where $M>m$, then there exist $M$ cars $k,k+1,\ldots,k+M-1$ with decreasing ranks. By Remark \ref{remark:1}, the $M$ cars are in different columns, which is impossible, thus the assumption that $M\in\pides(\PF)$ is not true.
\end{proof}

In many papers (e.g.\ \cite{GLWX,Em}), the statistic dinv of a parking function is defined by 3 components --- path dinv (pdinv), max dinv (maxdinv) and temporary dinv (tdinv). 
\begin{definition}[$\tdinv$] Let $\PF$ be any $(m, n)$-parking function, then 
	$$\tdinv(\PF): =\sum_{\mathrm{cars}\ i<j} \chi(\rank(i) < \rank(j) < \rank(i) + m).$$
\end{definition}
In \fref{2} (c), $\tdinv(\PF)=7$ since the pairs of cars contributing to tdinv are $(1,3)$, $(1,4)$, $(3,5)$, $(3,6)$, $(4,6)$, $(5,7)$ and $(6,7)$. Then, the statistic max dinv of a path is defined as the maximum of temporary dinvs of parking functions on the path.
\begin{definition}[$\maxdinv$] For any parking function $\PF$,
	$$\maxdinv(\PF) := \mathrm{max}\{\tdinv(\PF') : \Pi(\PF') =\Pi(\PF)\}.$$
\end{definition}
Finally, the statistic dinv is defined as follows.
\begin{definition}[$\dinv$]\label{definition:dinv1}
	For any parking function $\PF$,
	$$\dinv(\PF) := \tdinv(\PF)+\pdinv(\Pi(\PF))-\maxdinv(\PF).$$
\end{definition}
We shall apply this definition of dinv in several combinatorial proofs in Section 3. 

Notice that the statistics pdinv and maxdinv of a parking function $\pi$ are determined by the underlying Dyck path $\Pi(\PF)$. we also write $\pdinv(\PF)$ and $\maxdinv(\Pi(\PF))$ for path dinv of parking function $\PF$ and max dinv of its path. 

Further, the component $(\pdinv(\Pi(\PF))-\maxdinv(\PF))$ in the definition of $\dinv(\PF)$ combines to a statistic of rational Dyck paths.
Our definition of $\dinv(\PF)$ in Section 4 will follow the formulation by Leven and Hicks \cite{Hicks}, who gave a simplified formula for $\dinv(\PF)$ by defining the statistic \emph{dinv correction} (\emph{dinvcorr}) that satisfies $\dinvcorr(\Pi(\PF))=\pdinv(\Pi(\PF))-\maxdinv(\Pi(\PF))$. 

\begin{definition}[$\dinvcorr$] Let $P$ be any $(m, n)$-Dyck path and  set $\frac{0}{0} = 0$ and $\frac{x}{0} = \infty$ for all $x\neq 0$, then\ 
	
\noindent\resizebox{\textwidth}{!}{$
		\ \ \dinvcorr(P):=\sum_{c\in\lambda(P)}\chi\left( \frac{arm(c)+1}{leg(c)+1}\leq\frac{m}{n}<\frac{arm(c)}{leg(c)} \right)-\sum_{c\in\lambda(P)}\chi\left( \frac{arm(c)}{leg(c)}\leq\frac{m}{n}<\frac{arm(c)+1}{leg(c)+1} \right).
		$}
\end{definition}
An alternative definition of \emph{dinv} is 
\begin{definition}[$\dinv$, alt.]
	Let $\PF$ be any $(m, n)$-parking function, then 
	$$
	\dinv(\PF) := \tdinv(\PF)+\dinvcorr(\Pi(\PF)).
	$$
\end{definition}

Note that the statistic $\dinvcorr$ only depends on the path $P$, and it is the difference of two sums $\sum_{c\in\lambda(P)}\chi\left( \frac{arm(c)+1}{leg(c)+1}\leq\frac{m}{n}<\frac{arm(c)}{leg(c)} \right)$ and $\sum_{c\in\lambda(P)}\chi\left( \frac{arm(c)}{leg(c)}\leq\frac{m}{n}<\frac{arm(c)+1}{leg(c)+1} \right)$, of which at most one is nonzero. If $m=n$, then there is no $\dinvcorr$. If $m\neq n$, we count dinvcorr by all the cells in $\lambda(P)$. Given a cell $c\in\lambda(P)$, we highlight the vertical line segment $N$ which is a north step of the path $P$ to the east of $c$, and the horizontal line segment $E$ which is a east step of $P$ to the south of $c$. We draw two lines with slope $\frac{n}{m}$ from the north end and south end of $N$. 
\begin{enumerate}[(1)]
	\item If $n > m$, the cells of type (a) and (b) in \fref{dinvcorr} contribute $-1$ to $\dinvcorr$,
	\item If $m > n$, the cells of type (c) and (d) in \fref{dinvcorr} contribute $1$ to $\dinvcorr$.
\end{enumerate}

\begin{figure}[ht!]
	\centering
	\vspace{-1mm}	
	\begin{tikzpicture}[scale=.5]
	\draw (0,5) rectangle (1,6); 
	\node at (.5,5.5) {$c$};
	\draw[very thick] (0,1) circle[radius=0.05,fill]--(1,1) circle[radius=0.05,fill];
	\draw[very thick] (3,5) circle[radius=0.05,fill]--(3,6) circle[radius=0.05,fill];
	\draw (.5,0)--(3,5)--(3,6)--(0,0);
	\draw (1,-1) node {(a)};
	\fillll{.5}{2}{E};\fillll{4}{6}{N};
	\path (-2,0);
	\end{tikzpicture}
	\begin{tikzpicture}[scale=.5]
	\draw (0,5) rectangle (1,6); 
	\node at (.5,5.5) {$c$};
	\draw[very thick] (0,1) circle[radius=0.05,fill]--(1,1) circle[radius=0.05,fill];
	\draw[very thick] (3,5) circle[radius=0.05,fill]--(3,6) circle[radius=0.05,fill];
	\draw (.25,0)--(3,5)--(3,6)--(-.3,0);
	\draw (1,-1) node {(b)};
	\fillll{.5}{2}{E};\fillll{4}{6}{N};
	\path (-2,0);
	\end{tikzpicture}
	\begin{tikzpicture}[scale=.5]
	\draw (0,2) rectangle (1,3); 
	\node at (.5,2.5) {$c$};
	\draw[very thick] (0,1) circle[radius=0.05,fill]--(1,1) circle[radius=0.05,fill];
	\draw[very thick] (4,2) circle[radius=0.05,fill]--(4,3) circle[radius=0.05,fill];
	\draw (1,0.5)--(4,2)--(4,3)--(-1,.5);
	\path (-3,0);
	\fillll{1}{1}{E};\fillll{5}{3}{N};
	\draw (1.5,-1.5) node {(c)};
	\end{tikzpicture}
	\begin{tikzpicture}[scale=.5]
	\draw (0,2) rectangle (1,3); 
	\node at (.5,2.5) {$c$};
	\draw[very thick] (0,.75) circle[radius=0.05,fill]--(1,.75) circle[radius=0.05,fill];
	\draw[very thick] (4,2) circle[radius=0.05,fill]--(4,3) circle[radius=0.05,fill];
	\draw (1,0.5)--(4,2)--(4,3)--(-1,.5);
	\fillll{1}{.75}{E};\fillll{5}{3}{N};
	\draw (1.5,-1.5) node {(d)};
	\end{tikzpicture}
	\caption{Types of cells that contribute to $\dinvcorr$.}
	\label{fig:dinvcorr}
\end{figure}
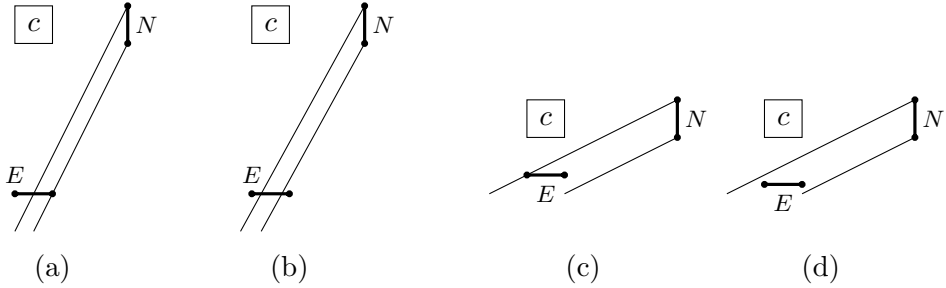

A  composition $\al$ of $n$ is a sequence of positive integers summing up to $n$, denoted by $\al\models n$.
Suppose $\alpha = (\alpha_1, \ldots, \alpha_k)$ is 
a composition of $n$ with $k$ parts. We associate 
a subset $S(\alpha)$ of $\{1, \ldots, n-1\}$ to $\alpha$ 
by setting 
$$S(\alpha) = \{\alpha_1, \alpha_1+\alpha_2, \ldots , \alpha_1 + \cdots + 
\alpha_{k-1}\}.$$
We let $F_{\al}[X]$ denote the fundamental quasi-symmetric function of Gessel \cite{gessel} associated to $\alpha$ where $X =\{x_1, \ldots, x_n\}$:
\begin{equation}
F_{\alpha}[X] := 
\sum_{\stackrel{1 \leq a_1 \leq a_2 \leq \cdots \leq a_n \leq n}{i \in 
		S(\alpha) \rightarrow a_i < a_{i+1}}} x_{a_1} x_{a_2} \cdots x_{a_n}.
\end{equation}
Then following Hikita \cite{Hikita}, the \textit{Hikita polynomial} 
$\Hikita_{m,n}[X; q, t]$ where $m$ and $n$ are coprime is defined by 
\begin{equation}
\Hikita_{m,n}[X; q, t]:=\sum_{\PF\in\PFc_{m,n}} t^{\area(\PF)}q^{\dinv(\PF)}F_{\pides(\PF)}[X].
\end{equation}
Due to the work of Haglund, Haiman, Loehr, Remmel and Ulyanov 
\cite{HHLRU},
the refinement of Hikita polynomial on each Dyck path is symmetric over the variables in $\{x_1,\ldots,x_n\}$, thus Hikita polynomials are symmetric functions.

Hikita did not define non-coprime Hikita polynomials, thus we generalize Hikita polynomials to non-coprime case as follows. Given $m,n$ coprime and $k \geq 1$, we defined the return, $\ret(\PF)$, of 
a $(km,kn)$-parking function $\PF$ to be the \textit{smallest} positive integer $i$ such that the supporting path of $\PF$ goes through the point $(im, in)$. 
Then following the formulation of Bergeron, Garsia, Leven and Xin \cite{BGLX}, the \emph{extended Hikita polynomial} is defined to be
\begin{equation}
\Hikita_{km,kn}[X; q, t]:=\sum_{\PF\in\PFc_{km,kn}} [ret(\PF)]_{\frac{1}{t}}t^{\area(\PF)}q^{\dinv(\PF)}F_{\pides(\PF)}[X].
\end{equation}
For the same reason, extended Hikita polynomials are also symmetric functions in $X$.

\subsection{Algebraic side}

For any partition $\mu$ of $n$, let $\widetilde{H}_{\mu}$ be the modified Macdonald symmetric function \cite{Macbook} associated to $\mu$, and let $\nabla$ be the linear operator defined in terms of the
modified Macdonald symmetric functions $\widetilde{H}_{\mu }(X;q,t)$ by
\begin{equation}\label{e:nabla}
\nabla \widetilde{H}_{\mu } := t^{n(\mu )}q^{n(\mu ')} \widetilde{H}_{\mu },
\end{equation}
where $\mu '$ is the conjugate of $\mu$, and
$n(\mu ) = \sum _{i}(i-1)\mu _{i}$. 

The Shuffle Conjecture proposed by Haglund, Haiman, Loehr, Remmel and Ulyanov 
\cite{HHLRU} which was proved by Carlsson and Mellit \cite{CM} as the Shuffle Theorem can be stated as follows.
\begin{theorem}[Carlsson-Mellit]For all $n\geq 0$, 
	\begin{equation}
	\nabla e_n=\Hikita_{n+1,n}[X; q, t].
	\end{equation}
\end{theorem}

Gorsky and Negut\cite{GN} introduced operators $\Qmn{m,n}$ on symmetric 
functions in the case where $m$ and $n$ are coprime. 
The $\Qmn{m,n}$ operators of the Gorsky-Negut  can be defined in terms of the operators $D_k$ which were introduced by Bergeron and Garsia \cite{GB}. 
In the plethystic notation, the action of $D_k$ on a symmetric function $F[X]$ is defined by
\begin{equation}
D_k\ F[X]=F\left[X+\frac{M}{z}\right]\sum_{i\geq 0}  (-z)^i e_i[X]\bigg\vert_{z^k},
\end{equation}
where $M = (1-t)(1-q)$.	

Then one can construct a family of symmetric function operators $\Qmn{m,n}$ for any pair of coprime positive integers $(m,n)$ as follows. First for any $n\geq 0$, set $\Qmn{1,n} = (-1)^n D_n.$ Next, one can recursively define $\Qmn{m,n}$ for $m > 1$ as follows. Consider the $m\times n$ lattice with diagonal $y =\frac{n}{m}x$. We choose $(a, b)$ such that $(a, b)$ is the lattice point which is closest to the diagonal, and 
\begin{equation*}
\begin{vmatrix} c & d \\ a & b \end{vmatrix} >0
\end{equation*}
where $(c, d) = (m -a, n -b)$. In such a case, we will write 
\begin{equation}
\mathrm{Split}(m, n) = (a, b) + (c, d).
\end{equation}

Note that the pairs $(a, b)$ and $(c,d)$ are coprime since any point of the form $(kx, ky)$ is further from the diagonal than the point $(x, y)$. Then we have the following recursive definition of the $\Qmn{m,n}$ operators:
\begin{equation}
\Qmn{m,n} = \frac{1}{M}[\Qmn{c,d}, \Qmn{a,b}] = \frac{1}{M}(\Qmn{c,d} \Qmn{a,b}-\Qmn{a,b}\Qmn{c,d}).
\end{equation}

\fref{3} gives an example of $\mathrm{Split}(3,5)$. $\mathrm{Split}(3,5) = (1,2) + (2,3)$, so that
\begin{equation}
\Qmn{3,5} = \frac{1}{M}[\Qmn{2,3}, \Qmn{1,2}]=\frac{1}{M}[\Qmn{2,3},D_2].
\end{equation}
\begin{figure}[ht!]
	\centering	
	\begin{tikzpicture}[scale=.8]
	\draw[help lines] (0,0) grid (3,5);
	\draw[->,green!50!black,very thick] (0,0) -- (3,5); \node[green!50!black] at (2,2.5) {\footnotesize $\Qmn{3,5}$};
	\draw[->,blue,thick] (0,0) -- (1,2); \node[blue] at (.4,1.8) {\footnotesize $\Qmn{1,2}$};
	\draw[->,blue,thick] (1,2) -- (3,5); \node[blue] at (1.5,3.5) {\footnotesize $\Qmn{2,3}$};
	\draw[->,orange!80!red,thick] (2,4) -- (3,5); \node[orange!80!red] at (2.2,4.7) {\footnotesize $\Qmn{1,1}$};
	\draw[->,orange!80!red,thick] (1,2) -- (2,4); \node[orange!80!red] at (1,3) {\footnotesize $\Qmn{1,2}$};
	\end{tikzpicture}
	\vspace{-2mm}
	\caption{The geometry of $\mathrm{Split}(3,5)$.}
	\label{fig:3}
\end{figure}
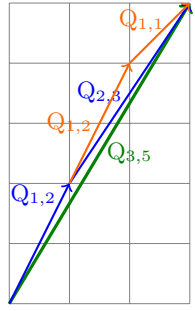

The same procedure gives $\Qmn{2,3} = \frac{1}{M}[\Qmn{1,1},\Qmn{1,2}]=\frac{1}{M}[-D_1, D_2]$. Therefore,
\begin{equation}
\Qmn{3,5} = \frac{1}{M^2}[[-D_1, D_2],D_2] = \frac{1}{M^2} (-D_2D_2D_1 + 2D_2D_1D_2 -D_1D_2D_2) .
\end{equation}

For the non-coprime case, we can define the $\Qmn{km,kn}$ operator as follows. We choose one of the lattice points, $(a,b)$, in the $km\times kn$ lattice satisfying $b(km-a)-a(km-b)>0$ that are not on the diagonal and closest to the diagonal,
then we set
\begin{equation}
\Qmn{km,kn} = \frac{1}{M}[\Qmn{km-a,kn-b}, \Qmn{a,b}].
\end{equation}
This recursive definition is well-defined as it is proved in \cite{BGLX} that any choice of such point $(a,b)$ defines the same operation.

Gorsky and Negut \cite{GN} have the following lemma about the operators  $\nabla$ and $\Qmn{m,n}$ for the coprime case.
\begin{lemma}[Gorsky-Negut]\label{lemma:1}
	For any coprime positive integers $m,n$,
	\begin{equation}
	\nabla \Qmn{m,n}\nabla^{-1}=\Qmn{m+n,n}.
	\end{equation}
\end{lemma}
We can generalize \lref{1} to the case $(kn,n)$.
\begin{lemma}\label{lemma:newlemma}
	For any positive integers $k,n$,
	\begin{equation}
	\nabla \Qmn{kn,n}\nabla^{-1}=\Qmn{(k+1)n,n}.
	\end{equation}
\end{lemma}
\begin{proof}
	By definition of the operator $\Qmn{m,n}$, we have
	\begin{equation*}
	\Qmn{kn,n}=\frac{1}{M}[\Qmn{kn-1,n},\Qmn{1,0}]
	=\frac{1}{M}(\Qmn{kn-1,n}\Qmn{1,0} -\Qmn{1,0}\Qmn{kn-1,n}).
	\end{equation*}
	Using \lref{1}, we have
	\begin{align*}
		\nabla \Qmn{kn,n}\nabla^{-1}
		=&\frac{1}{M}(\nabla\Qmn{kn-1,n}\nabla^{-1}\nabla\Qmn{1,0}\nabla^{-1} -\nabla\Qmn{1,0}\nabla^{-1}\nabla\Qmn{kn-1,n}\nabla^{-1})\\
		=&\frac{1}{M}(\Qmn{kn+n-1,n}\Qmn{1,0} -\Qmn{1,0}\Qmn{kn+n-1,n})\\
		=&\Qmn{(k+1)n,n}.\hspace*{8cm}\qedhere
	\end{align*}
\end{proof}

\bigskip

Now we are ready to prove \tref{2k2}.

\begin{proof}[Proof of \tref{2k2}]
One can verify the base case for both equations: 
\begin{equation}
\Qmn{2,2}(1)=s_2+(\sqt{1}+1)s_{11}.
\end{equation}
Note that by \lref{newlemma}, $\Qmn{2k,2} = \nabla^{k-1}\Qmn{2,2}\nabla^{-k+1}$, so we have
\begin{equation}
\Qmn{2k,2}(1)=\nabla^{k-1}\Qmn{2,2}\nabla^{-k+1}(1) =\nabla^{k-1}(s_2+(\sqt{1}+1)s_{11}).
\end{equation}
It can be proved by induction that 
\begin{align}
\nabla^{n}s_2 &=  -qt \qtn{n-1} s_2 - qt \qtn{n} s_{11},\\
\nabla^{n}s_{11}&=  \qtn{n} s_2 + qt \qtn{n+1} s_{11},
\end{align}
thus,
\begin{equation}
\Qmn{2k,2}(1)=\nabla^{k-1} s_2+(q+t+1)\nabla^{k-1} s_{11} = (\qtn{k}+\qtn{k-1})s_2 + (\qtn{k+1}+\qtn{k})s_{11},
\end{equation}
which proves Equation (\ref{eq1}).

For Equation (\ref{eq2}), we use the result of Leven (Equation (23) of \cite{Em}) that
\begin{equation}
D_a D_b (1) = (-1)^{a+b} e_a e_b +(-1)^{a+b-1} M \sum_{i=1}^{b} \qtn{i}  e_{a+i} e_{b-i}.
\end{equation}
Expanding the operator $\Qmn{2,2k}$ gives
\begin{align}
\Qmn{2,2k}(1)& = \frac{1}{M} (D_{k+1}D_{k-1}(1)- D_{k-1}D_{k+1}(1))\nonumber\\
& = \sum_{i=1}^{k+1} \qtn{i}  e_{k+1-i} e_{k-1+i} 
-\sum_{i=1}^{k-1} \qtn{i}  e_{k-1-i} e_{k+1+i} \nonumber\\
& = \sum_{r=0}^k (\qtn{k-r+1}+\qtn{k-r})s_{2^r1^{2k-2r+1}},
\end{align}
which proves Equation (\ref{eq2}).
\end{proof}

We shall use \lref{1} and \lref{newlemma} to prove
\tref{1} algebraically in Section 4.1.

\subsection{The Extended Rational Shuffle Theorem}
For the rational case, we consider pairs of positive integers $(km,kn)$, where $m$ and $n$ are coprime and $k$ is a positive integer.
The \emph{Extended Rational Shuffle Conjecture} of Bergeron-Garsia-Leven-Xin \cite{BGLX}, which generalizes a previous conjecture by Gorsky and Negut \cite{GN}, has been shown to hold by Mellit \cite{Mellit}. So we have the \emph{Extended Rational Shuffle Theorem} as follows.
\begin{theorem}[Mellit]\label{theorem:GLWX}
	For all pairs of coprime positive integers $(m,n)$ and all $k\in\mathbb{Z}^+$, we have
	\begin{equation}
	\Qmn{km,kn}(1)=\Hikita_{km,kn}[X; q, t].
	\end{equation}
\end{theorem}

The original \emph{Rational Shuffle Theorem} proposed by Gorsky and Negut \cite{GN} and proved by  Mellit \cite{Mellit} in the case where $m$ and 
$n$ are relatively prime is the special case when $k=1$ in \tref{GLWX}.
Further details and the complete picture about the Rational Shuffle Theorem is outlined in \cite{B2}.
Notice that $\nabla e_n=\Qmn{n+1,n}(1)$, and $\nabla e_n$ is not the same as $\Qmn{n,n}(1)$. 


The main goal of this paper is to study the combinatorics of 
the Schur function expansion of $\Qmn{m,n} (1)$.
Given that Mellit has proved the Rational Shuffle Theorem, we can 
find the Schur function expansion in one of two ways. That is, we can 
use the properties of $\Qmn{m,n}$ to find the Schur function expansion 
of $\Qmn{m,n} (1)$ which we will refer to as working on the {\em symmetric 
	function side} of the Rational Shuffle Theorem. Second, one could 
start with the Hikita polynomial $\Hikita_{m,n}[X;q,t]$ and expand 
that polynomial into Schur functions which we will call working 
on the {\em combinatorial side} of the Rational Shuffle Theorem. 

Since it is proved that $\Qmn{m,n}(1)=\Hikita_{m,n}[X; q, t]$, we let $[s_\lambda]_{m,n}$ denote the coefficient of Schur function $s_{\lambda}$ in both polynomials $\Qmn{m,n}(1)$ and $\Hikita_{m,n}[X; q, t]$. 

\subsection{An alternative expression for the combinatorial side}
We shall introduce an alternative expression for Hikita polynomials due to the fact that the Hikita polynomials $\Hikita_{m,n}[X;q,t]$ are symmetric in $\{x_1, \ldots, x_n\}$.

Suppose that  
$\al = (\al_1, \ldots, \al_k)$ is a  composition of 
$n$ into $k$ parts ($k\leq n$), then we set $\al_j=0$ for $j>k$. We let $X =\{x_1, \ldots, x_n\}$ be the set of $n$ variables and 
\begin{equation*}
\Delta_{\al}[X] :=  \det||x_i^{\al_j+n-j}|| =  \sum_{\sg  \in S_n} \mathrm{sgn}(\sg) \sg(x_1^{\al_1+n-1} \cdots 
x_n^{\al_n+n-n}).
\end{equation*}
We let $\displaystyle \Delta[X] := \det||x_i^{n-j}||$  
be the Vandermonde determinant. The Schur symmetric function 
$s_{\al}[X]$ 
associated to $\al$ can be defined by
\begin{math}\displaystyle
s_{\al}[X] := \frac{\Delta_{\al}[X]}{\Delta[X]}.
\end{math}

It is well-known that for any such  composition $\al$,  
either we have $s_{\al}[X] =0$ or there is a 
partition $\lambda\vdash n$ such that 
$s_{\al}[X] = \pm s_{\lambda}[X]$.
In fact, there is a  \emph{straightening} relation 
which allows us to prove that fact. Namely, 
if $\al_{i+1}  > \al_i$, then 
\begin{equation}\label{straight}
s_{(\al_1, \ldots, \al_i, \al_{i+1}, \ldots, \al_k)}[X] = - s_{(\al_1, \ldots, \al_{i+1}-1, \al_{i}+1, \ldots, \al_k)}[X].
\end{equation}

In a remarkable and important paper, Egge, Loehr and Warrington  
\cite{ELW} gave a combinatorial description of how 
to start with a quasi-symmetric function expansion of a homogeneous 
symmetric function $P[X]$ of degree $n$, 
and transform it into an expansion in terms of Schur functions.
The following theorem due to Garsia and Remmel \cite{GR} is implicit 
in the work of \cite{ELW}, but is not explicitly stated and it allows 
one to find the Schur function expansion by using the straightening laws. 
\begin{theorem}[Garsia-Remmel]\label{theorem:main}
	Suppose that $P[X]$ is a symmetric function which is 
	homogeneous of degree $n$ and 
	\begin{equation}\label{Pquasi}
	P[X] = \sum_{\alpha \vDash n} a_{\alpha} F_\alpha[X].
	\end{equation}
	Then 
	\begin{equation}\label{Pschur}
	P[X] = \sum_{\alpha \vDash n} 
	a_{\alpha} s_{\alpha}[X].
	\end{equation}
\end{theorem}

Recall that $\pides(\sigma)$ is the composition set of $\ides(\sigma)$, then \tref{main} and the 
straightening action allow us to transform $\Hikita_{m,n}[X;q,t]$ into Schur function expansion that
\begin{align}
\Hikita_{m,n}[X; q, t]&=\sum_{\PF\in\PFc_{m,n}} [\ret(\PF)]_{\frac{1}{t}} t^{\area(\PF)}q^{\dinv(\PF)}F_{\pides(\PF)}\nonumber\\
&=\sum_{\PF\in\PFc_{m,n}} [\ret(\PF)]_{\frac{1}{t}} t^{\area(\PF)}q^{\dinv(\PF)}s_{\pides(\PF)}.\label{Hikitas}
\end{align}

From Section 3, we shall use the expression (\ref{Hikitas}) for $\Hikita_{m,n}[X; q, t]$ to prove several
facts about the coefficients in the Schur function expansions of the Rational Shuffle Theorem.

\section{Combinatorial results about Schur function expansions of the $(m,n)$ case}

We shall work on the combinatorial side by studying the Hikita polynomials in this section, and we use the expression of Hikita polynomials in Equation (\ref{Hikitas}).

In the rational $(m,n)$ case, we have $n$ cars, i.e.\ the word of an $(m,n)$-parking function is a permutation of $[n]=\{1,\ldots,n\}$. 
Recall that $\scoeff{\lambda}_{m,n}$ is the coefficient of $s_\la$ in $\Hikita_{m,n}[X; q, t]$.
By Remark \ref{remark:2} in Section 2.1, $[s_{\lambda}]_{m,n} \neq 0$ implies that 
$\lambda$ must be of the form $m^{\alpha_m}\cdots1^{\alpha_1}$ with $\sum_{i=1}^{m}i\alpha_i=n$, i.e.\ $[s_{\lambda}]_{m,n} \neq 0$ only if the partition $\lambda$ has parts of size less than or equal to $m$. In this section, we shall prove the 3 symmetries about $\scoeff{\lambda}_{m,n}$ described in \tref{3}, stated as the following three lemmas.

\begin{result}
	$[s_{1^n}]_{m,n}=[s_n]_{m+n,n}$.
\end{result}

Note that a parking function with $\pides$ $n$ must have word $12\cdots n$, and a parking function with $\pides$ $1^n$ must have word $n\cdots 21$. 

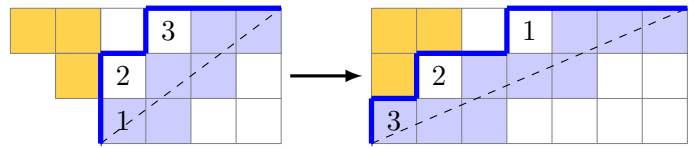
\begin{figure}[ht!]
	\centering	
	\begin{tikzpicture}[scale=0.6]
	\fillshadea{0/2,0/3,-1/3,7/2,7/3,8/3}
	\draw[help lines] (-2,3) grid (0,2) grid (-1,1);
	\fillshade{1/1,2/1,2/2,3/2,3/3,4/3,7/1,8/1,9/1,9/2,10/2,11/2,11/3,12/3,13/3};
	\PFmnu{0,0}{4}{3}{0/1,0/2,1/3,0/0};
	\arrowx{4.2,1.5}{5.8,1.5};
	\PFmnu{6,0}{7}{3}{0/3,1/2,3/1,0/0};
	\end{tikzpicture}
	\caption{Bijection between $\PFc_{m,3}$ with word $123$ and $\PFc_{m+3,3}$ with word $321$.}
	\label{fig:7}	
\end{figure}

A parking function in $\PFc_{m,n}$ with word $n\cdots 21$ corresponds to a unique $(m,n)$-Dyck path, and a parking function in $\PFc_{m+n,n}$ with word $12\cdots n$ corresponds to a unique $(m+n,n)$-Dyck path with no consecutive north steps. As shown is \fref{7}, we can obtain a parking function in $\PFc_{m+n,n}$ with word $12\cdots n$ by pushing a staircase into a parking function $\PF\in\PFc_{m,n}$ with word $n\cdots 21$. Given  $\PF\in\PFc_{m,n}$ with word $n\cdots 21$, let $\lambda=\lambda(\PF)$, we define $\hstr(\PF)\in\PFc_{m+n,n}$, the {\em horizontal stretch} of $\PF$, to be the parking function with word $12\cdots n$ and $\lambda(\hstr(\PF))=(\lambda_1+n-1,\lambda_2+n-2,\ldots,\lambda_{n-1}+1)$, then
\begin{theorem}\label{theorem:comb1}
	\begin{align*}
		\hstr:\ \{\PF\in\PFc_{m,n}:\word(\PF)=n\cdots 21\}&\rightarrow\{\PF\in\PFc_{m+n,n}:\word(\PF)=12\cdots n\},\\
		\PF&\mapsto \hstr(\PF)
	\end{align*}
	is a bijection, and 
	\begin{align}
	\label{harea}\area(\hstr(\PF))&=\area(\PF),\\
	\label{hdinv}\dinv(\hstr(\PF))&=\dinv(\PF),\\
	\label{hret}\ret(\hstr(\PF))&=\ret(\PF).
	\end{align}
\end{theorem}
\begin{proof}
	The bijectivity of the map $\hstr$ is clear since the map is invertible. Comparing the coarea of both parking functions immediately proves Equations (\ref{harea}) and (\ref{hret}). To prove Equation (\ref{hdinv}), recall that $\dinv(\PF)=\tdinv(\PF)+\dinvcorr(\PF)$, we shall compare the two components of dinv, i.e.\ tdinv and dinvcorr.
	
	For a parking function $\PF\in\PFc_{m,n}$ with $\word(\PF)=n\cdots 21$, its temporary dinv reaches the maximum possible value of its path $\Pi(\PF)$, i.e.\ any two north steps with rank difference less than $m$ will contribute 1 to tdinv.
	For any two north steps, we fire two lines parallel to the diagonal from the two end points of the upper north step, then rank difference less than $m$ means that either the upper line or the lower line intersects the lower north step. The two cases are pictured in \fref{tdinv123}.
	
	On the other hand, the parking function $\hstr(\PF)\in\PFc_{m+n,n}$ 
	always has no tdinv since \\$\word(\hstr(\PF))=12\cdots n$. We shall show that the increment of dinvcorr makes up for the missing tdinv.
	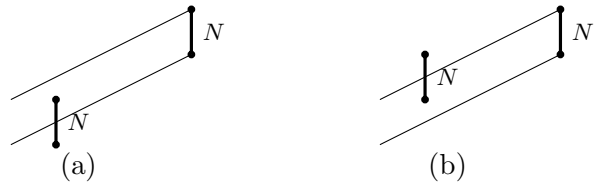
\begin{figure}[ht!]
		\centering	
		\begin{tikzpicture}[scale=0.6]
		\draw[very thick] (1,0) circle[radius=0.05,fill]--(1,1) circle[radius=0.05,fill];
		\draw[very thick] (4,2) circle[radius=0.05,fill]--(4,3) circle[radius=0.05,fill];
		\draw (0,0)--(4,2)--(4,3)--(0,1);
		\fillll{2}{1}{N};\fillll{5}{3}{N};
		\draw (1.5,-.5) node {(a)};
		\end{tikzpicture}
		\begin{tikzpicture}[scale=0.6]
		\draw[very thick] (1,1) circle[radius=0.05,fill]--(1,2) circle[radius=0.05,fill];
		\draw[very thick] (4,2) circle[radius=0.05,fill]--(4,3) circle[radius=0.05,fill];
		\draw (0,0)--(4,2)--(4,3)--(0,1);
		\path (-3,0);
		\fillll{2}{2}{N};\fillll{5}{3}{N};
		\draw (1.5,-.5) node {(b)};
		\end{tikzpicture}
		\caption{Pairs of north steps contributing to tdinv.}
		\label{fig:tdinv123}
	\end{figure}
	
	For a parking function $\PF\in\PFc_{m,n}$, suppose that there are $j$ cells in row $r$ of $\PF$ (counting from bottom to top) with leg $i$ in the English partition $\la(\PF)$, and their arms are $a,a+1,\ldots,a+j-1$, pictured in \fref{dinvchange} (a). We fire two lines with slope $\frac{n}{m}$ from the two end points of the north step (called $N_1$) in row $r$, then they intersect the east steps (called $EEs$) below the $j$ cells at points $A,B$ which have horizontal distances $\frac{mi}{n}$ and $\frac{m(i+1)}{n}$ to $N_1$.
	
	Now consider the parking function $\hstr(\PF)\in\PFc_{m+n,n}$. By definition of $\hstr$, there are $j+1$ cells in row $r$ with leg $i$ in the partition $\la(\hstr(\PF))$, and their arms are $a+i,a+i+1,\ldots,a+i+j$, pictured in \fref{dinvchange} (b). We again fire two lines with slope $\frac{n}{m+n}$ from the two end points of the north step $N_1$ in row $r$, then they intersect the east steps below the $j+1$ cells at points $A,B$ which have horizontal distances $\frac{mi}{n}+i$ and $\frac{m(i+1)}{n}+i+1$ to $N_1$.
	\begin{figure}[ht!]
		\centering
		\begin{tikzpicture}[scale=0.6]
		\draw (0,0) grid (3,-1);
		\draw[ decorate,decoration={brace,amplitude=5pt,mirror},xshift=0.4pt,yshift=-0.4pt] (3,0) -- (0,0) node[midway,yshift=.5cm] {$j$ cells};
		\draw[ decorate,decoration={brace,amplitude=5pt,mirror},xshift=0.4pt,yshift=-0.4pt] (6,0) -- (3,0) node[midway,yshift=.5cm] {$a$};
		\draw[ decorate,decoration={brace,amplitude=5pt,mirror},xshift=0.4pt,yshift=-0.4pt] (0,-1) -- (0,-3) node[midway,xshift=-.5cm] {$i$};
		\draw[ultra thick] (0,-3) -- (3,-3) -- (3,-2);
		\node at (3.5,-2.5) {$N_2$};
		\draw[ultra thick, dashed] (3,-2) -- (6,-1);
		\draw[ultra thick] (6,-1) -- (6,0);
		\node at (6.5,-.5) {$N_1$};
		\node at (-1.3,-.5) {row $r$};
		\node at (1.5,-3.5) {$EEs$};
		
		\draw (6,-1) -- (2.5,-3) -- (.7,-3) -- (6,0);
		\draw[help lines] (2.5,-3) -- (2.5,-5);
		\draw[help lines] (.7,-3) -- (.7,-6);
		\draw[help lines] (6,-1) -- (6,-6);
		\draw[ decorate,decoration={brace,amplitude=5pt,mirror},xshift=0.4pt,yshift=-0.4pt] (0.7,-5.5) -- (6,-5.5) node[midway,yshift=-.5cm] {$\frac{m(i+1)}{n}$};
		\draw[ decorate,decoration={brace,amplitude=5pt,mirror},xshift=0.4pt,yshift=-0.4pt] (2.5,-4) -- (6,-4) node[midway,yshift=-.5cm] {$\frac{mi}{n}$};
		\draw (3,-8) node {(a)};
		\node at (.7,-2.5) {$A$};
		\node at (2.5,-2.5) {$B$};
		\path (9,0);
		\end{tikzpicture}
		\begin{tikzpicture}[scale=0.6]
		\draw (0,0) grid (4,-1);
		\draw[ decorate,decoration={brace,amplitude=5pt,mirror},xshift=0.4pt,yshift=-0.4pt] (4,0) -- (0,0) node[midway,yshift=.5cm] {$j+1$ cells};
		\draw[ decorate,decoration={brace,amplitude=5pt,mirror},xshift=0.4pt,yshift=-0.4pt] (8,0) -- (4,0) node[midway,yshift=.5cm] {$a+i$};
		\draw[ decorate,decoration={brace,amplitude=5pt,mirror},xshift=0.4pt,yshift=-0.4pt] (0,-1) -- (0,-3) node[midway,xshift=-.5cm] {$i$};
		\draw[ultra thick] (0,-3) -- (4,-3) -- (4,-2);
		\node at (4.5,-2.5) {$N_2$};
		\draw[ultra thick, dashed] (4,-2) -- (8,-1);
		\draw[ultra thick] (8,-1) -- (8,0);
		\node at (8.5,-.5) {$N_1$};
		\node at (-1.3,-.5) {row $r$};
		\node at (2,-3.5) {$EEs$};
		
		\draw (8,-1) -- (3.5,-3) -- (.7,-3) -- (8,0);
		\draw[help lines] (3.5,-3) -- (3.5,-5);
		\draw[help lines] (.7,-3) -- (.7,-6);
		\draw[help lines] (8,-1) -- (8,-6);
		\draw[ decorate,decoration={brace,amplitude=5pt,mirror},xshift=0.4pt,yshift=-0.4pt] (0.7,-5.5) -- (8,-5.5) node[midway,yshift=-.5cm] {$\frac{m(i+1)}{n}+i+1$};
		\draw[ decorate,decoration={brace,amplitude=5pt,mirror},xshift=0.4pt,yshift=-0.4pt] (3.5,-4) -- (8,-4) node[midway,yshift=-.5cm] {$\frac{mi}{n}+i$};
		\node at (.7,-2.5) {$A$};
		\node at (3.5,-2.5) {$B$};
		\draw (4,-8) node {(b)};
		\end{tikzpicture}
		\caption{Cells in row $r$ with leg $i$.}
		\label{fig:dinvchange}
	\end{figure}
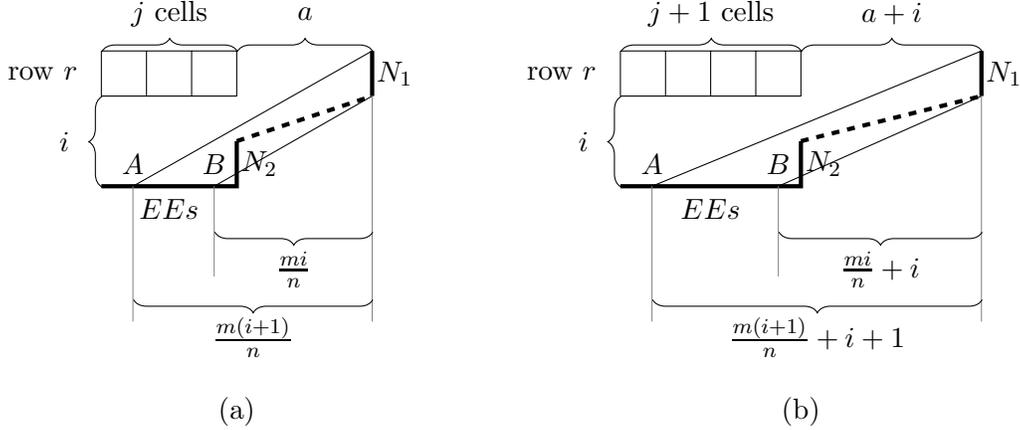
	
	Now recall the definition of the dinv correction. The dinvcorr contribution of $N_1$ in each picture is equal to the whole east steps contained in line segment $\overline{AB}$. The line segment $\overline{AB}$ in $\hstr(\PF)$ contains one more east step than $\overline{AB}$ in $\PF$ in the following 2 cases:
	\begin{enumerate}[(1)]
		\item In $\PF$, $A$ is not on $EEs$ but $B$ is on $EEs$.
		\item In $\PF$, $A$ is on $EEs$.
	\end{enumerate}
	In case (1), the car in row $r$ of $\PF$ produces a tdinv with the car in the row immediately below $EEs$; in case (2), the car in row $r$ of $\PF$ produces a tdinv with the car in the row of the next north step that the upper line fired from $N_1$ intersects. Thus, the new dinvcorr in case (1) and case (2) matches the tdinv in the two cases in \fref{tdinv123}, and the increment of dinv correction is equal to $\tdinv(\PF)$, which proves the theorem.
\end{proof}
Since $\hstr$ is an (area,dinv,ret)-preserving bijection, $[s_{1^n}]_{m,n}=[s_n]_{m+n,n}$ follows immediately.

\ 
\begin{result}
	$[s_{m^{\alpha_m}\cdots1^{\alpha_1}}]_{m,n}=[s_{m^{\alpha_m+1}\cdots1^{\alpha_1}}]_{m,n+m}$
\end{result}

This is a rewording of \tref{3} (b). 
For a parking function $\PF\in\PFc_{m,n}$, we define a map $\vstr$, {vertical stretch}, that we push a staircase down to $\PF$, then replace the car $i$ in $\PF$ by $i+m$, and fill the bottom of the $m$ columns of the new parking function with cars $1,\ldots,m$ in a rank decreasing way  to get $\vstr(\PF)$, as shown in \fref{8}.
\begin{figure}[ht!]
	\centering
	\begin{tikzpicture}[scale=0.6]
	\fillshadea{1/5,1/6,2/5}
	\fillshade{1/1,1/2,2/2,2/3,3/3,3/4};
	\draw[help lines] (2,5) grid (0,4) grid (1,6);
	\PFmnu{0,0}{3}{4}{0/2,0/3,1/1,1/4,0/0};
	\arrowx{3.7,2}{5.3,2};
	\path (6,0);
	\end{tikzpicture}
	\begin{tikzpicture}[scale=0.6]
	\fillshadea{1/6,1/7,2/7}
	\fillshade{1/1,1/2,1/3,2/3,2/4,2/5,3/5,3/6,3/7};
	\PFmnu{0,0}{3}{7}{0/{3},0/2,0/3,1/2,1/1,1/4,2/1,0/0};
	\fillshadesmall{1/1,2/4,3/7};
	\draw[dashed] (0,0) -- (3,7);
	\arrowx{3.7,2}{5.3,2};
	\path (6,0);
	\end{tikzpicture}
	\begin{tikzpicture}[scale=0.6]
	\fillshadea{1/6,1/7,2/7}
	\fillshade{1/1,1/2,1/3,2/3,2/4,2/5,3/5,3/6,3/7};
	\PFmnu{0,0}{3}{7}{0/{3},0/5,0/6,1/2,1/4,1/7,2/1,0/0};
	\fillshadesmall{1/1,2/4,3/7};
	\draw[dashed] (0,0) -- (3,7);
	\arrowx{3.7,2}{5.3,2};
	\path (6,0);
	\end{tikzpicture}
	\begin{tikzpicture}[scale=0.6]
	\fillshadea{1/6,1/7,2/7}
	\fillshade{1/1,1/2,1/3,2/3,2/4,2/5,3/5,3/6,3/7};
	\PFmnu{0,0}{3}{7}{0/{3},0/5,0/6,1/2,1/4,1/7,2/1,0/0};
	\end{tikzpicture}
	\caption{Bijection between $\PFc_{3,n}$ with $\pides\ 3^a2^b1^c$ and $\PFc_{3,n+3}$ with $\pides\ 3^{a+1}2^b1^c$.}
	\label{fig:8}	
\end{figure}
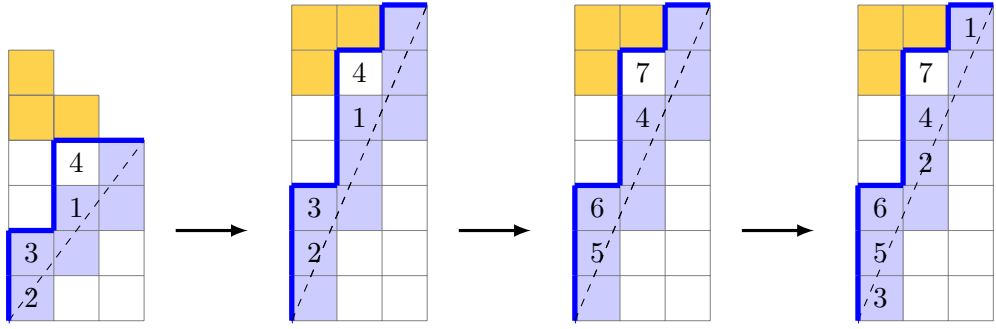

Similar to \tref{comb1}, we have the following theorem about the vertical stretch action.
\begin{theorem}\label{theorem:comb2}
	\begin{align*}
		\vstr:\ \{\PF\in\PFc_{m,n}:\pides(\PF)=m^{\alpha_m}\cdots1^{\alpha_1}\}&\rightarrow\{\PF\in\PFc_{m,n+m}:\pides(\PF)=m^{\alpha_m+1}\cdots1^{\alpha_1}\},\\
		\PF&\mapsto \vstr(\PF)
	\end{align*}
	is a bijection, and 
	\begin{align}
	\label{varea}\area(\vstr(\PF))&=\area(\PF),\\
	\label{vdinv}\dinv(\vstr(\PF))&=\dinv(\PF),\\
	\label{vret}\ret(\vstr(\PF))&=\ret(\PF).
	\end{align}
\end{theorem}
\begin{proof}
	The bijectivity follows from the invertibility of the map. Equations (\ref{varea}) and (\ref{vret}) are true for the same reason as Equations (\ref{harea}) and (\ref{hret}). The proof of Equation (\ref{vdinv}) is based on a similar idea to the proof of (\ref{hdinv}): the action of $\vstr$ changes each car $i$ in $\PF$ into $i+m$, and the rank is also increased by $m$, thus the temporary dinv of $\PF$ is equal to the temporary dinv of the cars $m+1,\ldots,m+n$ in $\vstr(\PF)$. Since the dinv correction is negative, we can match each tdinv between cars $1,2,\ldots,m$ and $m+1,\ldots,m+n$ with a new negative dinv correction, showing that the change of dinv is zero.
\end{proof}

\begin{result}
	$[s_{k1^{n-k}}]_{m,n}=[s_{k1^{m-k}}]_{n,m}$.
\end{result}

We shall prove the special case when $k=1$ first. That is, we first show $[s_{1^{n}}]_{m,n}=[s_{1^{m}}]_{n,m}$.
The bijection for this identity is that we can transpose the path of $\PF\in\PFc_{m,n}$ and fill the word $(m,m-1,\ldots,1)$ to get $\PF'\in\PFc_{n,m}$. 

It is obvious that $\PF'$ has the same $\area$ as $\PF$ since their underlying Dyck paths are transposes of each other. For the statistic dinv, recall that the tdinv of a parking function with word $(n,n-1,\ldots,1)$ is equal to the maxdinv of the path, thus
\begin{align}
\dinv(\PF)&=\tdinv(\PF)+\pdinv(\Pi(\PF))-\maxdinv(\PF)\nonumber\\
&= \pdinv(\Pi(\PF))\nonumber\\
&=\sum_{c\in\lambda(\Pi(\PF))}\chi\left(\frac{arm(c)}{leg(c)+1}\leq\frac{m}{n}<\frac{arm(c)+1}{leg(c)}\right).\label{transmn}
\end{align}
From the Equation (\ref{transmn}), we see that dinv is symmetric about $m$ and $n$, and preserved by the transpose action. Thus \fref{9} shows an example of this bijection.
\begin{figure}[ht!]
	\centering
	\begin{tikzpicture}[scale=0.6]
	\fillshade{1/1,2/1,2/2,3/2,3/3,4/3,7/1,7/2,8/2,8/3,9/3,9/4};
	\PFmnu{0,0}{4}{3}{0/1,0/2,1/3,0/0};
	\arrowx{4.2,1.5}{5.8,1.5};
	\PFmnu{6,0}{3}{4}{0/1,0/2,0/4,1/3,0/0};
	\end{tikzpicture}
	\caption{Bijection between $\PFc_{4,3}$ with $\pides\ 1^3$ and $\PFc_{3,4}$ with $\pides\ 1^4$.}
	\label{fig:9}	
\end{figure}
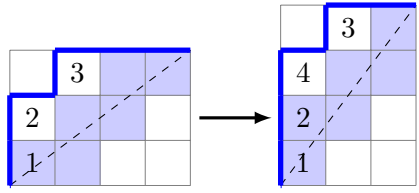

Then we consider the identity $[s_{k1^{n-k}}]_{m,n}=[s_{k1^{m-k}}]_{n,m}$.
This bijective proof is similar to that of $[s_{1^{n}}]_{m,n}=[s_{1^{m}}]_{n,m}$.

That is, given a parking function $\PF\in\PFc_{m,n}$ with pides ${k1^{n-k}}$, one transposes the path and labels the path to produce $\pides\ {k1^{m-k}}$. If there are only $k$ peaks (which means $k$ different columns) in the Dyck paths, then the filling of cars in both $(m,n)$ and $(n,m)$ cases are unique since the cars $1,\ldots,k$ must be filled in a rank-decreasing way at bottom of each column in the two parking functions, while the remaining cars should be filled in a rank-increasing way in the remaining north steps. One can check that they have the same $\area$ and $\dinv$ values. 

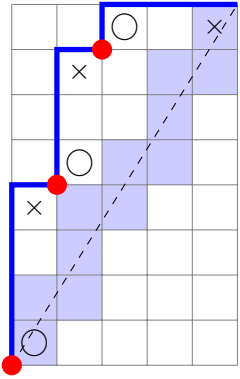
\begin{figure}[ht!]
	\centering
	\begin{tikzpicture}[scale=0.6]
	\fillshade{1/1,1/2,2/2,2/3,2/4,3/4,3/5,4/5,4/6,4/7,5/7,5/8};
	\draw[help lines] (0,0) grid (5,8);
	\draw[dashed] (0,0) -- (5,8);
	\draw[line width=2pt,blue] (0,0) --(0,4)--(1,4)--(1,7)--(2,7)--(2,8)-- (5,8);
	\fillcir{1/1,2/5,3/8};
	\fillcro{1/4,2/7,5/8};
	\fillpoint{0/0,1/4,2/7};
	\end{tikzpicture}
	\caption{A $(5,8)$-Dyck path with $3$ peaks.}
	\label{fig:dinvtrans}	
\end{figure}

Otherwise, in any rational $(m,n)$-Dyck path $P$ with $j>k$ peaks, the car $k$ must be in the first row since it has the smallest rank, and there are $\binom{j-1}{k-1}$ ways to choose columns for cars $1,\ldots,k-1$ in the north steps of both $P$ and its transpose $P'$, while the remaining cars should be filled in a rank-increasing way in the remaining north steps. We want to match the $\binom{j-1}{k-1}$ possible positions of cars $1,\ldots,k-1$ in both $(m,n)$ and $(n,m)$ cases by a similar idea.

We still use Definition \ref{definition:dinv1} as the definition of dinv, and the fact that a path $P$ and its transpose $P'$ have the same pdinv. For a parking function $\PF\in\PFc_{m,n}$ with pides ${k1^{n-k}}$ and a parking function $\PF'\in\PFc_{n,m}$ with pides ${k1^{m-k}}$, the component $(\tdinv(\PF)-\maxdinv(\PF))$ counts the missing dinv created by the first $k$ cars (since the cars greater than $k$ are filled in a way to generate maximum possible dinv).

Taking the $(5,8)$-Dyck path $P$ in \fref{dinvtrans} for an example. It has three peaks, and the three circles are the positions that the first $k$ cars can be filled in. Next we consider its transpose $P'$. We can use the same picture for the underlying Dyck path, and fill cars in its east steps. In this way, the three crosses are the positions that the first $k$ cars can be filled in. The rank of each cross is determined by the lattice point to the northeast of itself, and the rank of each circle is determined by the lattice point to the southwest of itself, and these lattice points are exactly the three valley points (including the start point) of the Dyck path, marked in the picture. Thus, each cross is paired with a circle by a certain valley point. 

For any parking function on $P$ with pides ${k1^{n-k}}$, we find the $k$ circled positions that contain the first $k$ cars, then we choose the corresponding $k$ crossed positions in the path $P'$. In this way of matching, the dinv component $(\tdinv(\PF)-\maxdinv(\PF))$ in two parking functions are the same, thus $[s_{k1^{n-k}}]_{m,n}=[s_{k1^{m-k}}]_{n,m}$ is proved.

\bigskip

Note that \tref{3} (c) is a result about hook-shaped Schur functions.
As we proved within this result, \tref{3} (c) implies the following corollary.
\begin{corollary}For all $m,n > 0$,
	$$[s_{1^n}]_{m,n}=[s_{1^m}]_{n,m}.$$
\end{corollary}

\section{Schur function expansions of the $(m,3)$ Case}

The Rational Shuffle Theorem when $n=3$ has a nice Schur function expansion, summarized in \tref{1}. For example, one can compute the Schur function expansion of $\Qmn{3k+1,3}(1)$ by Maple to get \taref{coeff}.

In this section, we give two proofs of \tref{1} by both 
working on the symmetric function side 
and the combinatorial side of the Rational Shuffle Theorem. Our proofs independently prove the Rational Shuffle Theorem and the Shuffle Theorem when $n\leq 3$.

\begin{table}[ht!]
	\centering
	\renewcommand{\arraystretch}{1.4}
	\caption{Coefficients of $s_{\lambda}$ in $\Qmn{3k+1,3}(1)$.}
	\vspace*{2mm}
	\begin{tabular}{|c|c|c|c|c|}
		\hline
		$k$ & {$\Qmn{3k+1,3}(1)$} & $\scoeff{3}_{3k+1,3}$ & $\scoeff{21}_{3k+1,3}$ & $\scoeff{111}_{3k+1,3}$ \\
		\hline
		$0$ & $\Qmn{1,3}(1)$ & $0$ & $0$  & $\sqt{\emptyset}$ \\\hline
		$1$ & $\Qmn{4,3}(1)$ & $\sqt{\emptyset}$ & $\sqt{1}+\sqt{2}$ & $\sqt{3}$\\
		&& &   & $+\sqt{1,1}$\\\hline
		$2$ & $\Qmn{7,3}(1)$ & $\sqt{3}$ & $\sqt{4}+\sqt{5}$ & $\sqt{6}$\\
		&&$+\sqt{1,1}$  & $+\sqt{2,1}+\sqt{3,1}$  & $+\sqt{4,1}$\\
		&&  &   & $+\sqt{2,2}$\\\hline
		$3$ & $\Qmn{10,3}(1)$ & $\sqt{6}$ & $\sqt{7}+\sqt{8}$ & $\sqt{9}$\\
		&&$+\sqt{4,1}$  & $+\sqt{5,1}+\sqt{6,1}$  & $+\sqt{7,1}$\\
		&& $+\sqt{2,2}$ &  $+\sqt{3,2}+\sqt{4,2}$ & $+\sqt{5,2}$\\
		&&  &   & $+\sqt{3,3}$\\\hline
		$4$ & $\Qmn{13,3}(1)$ & $\sqt{9}$ & $\sqt{10}+\sqt{11}$ & $\sqt{12}$\\
		&&$+\sqt{7,1}$  & $+\sqt{8,1}+\sqt{9,1}$  & $+\sqt{10,1}$\\
		&& $+\sqt{5,2}$ &  $+\sqt{6,2}+\sqt{7,2}$ & $+\sqt{8,2}$\\
		&& $+\sqt{3,3}$ & $+\sqt{4,3}+\sqt{5,3}$  & $+\sqt{6,3}$\\
		& &  &   & $+\sqt{4,4}$\\\hline
	\end{tabular}
	\label{table:coeff}
\end{table}

\subsection{Algebraic proof --- $\Qmn{m,3}(1)$}
We shall use Leven's method in \cite{Em} to prove the theorem by induction.
We use the following lemma about $(q,t)$-Schur functions to simplify our computation.
\begin{lemma}\label{lemma:2}Let $n,k\geq 0$ be two non-negative integers, we have
	\begin{equation}
	\sqt{n-1} \sqt{k-1} = \sqt{n+k-2} + \sqt{k-1,1} \sqt{n-2}.
	\end{equation}
\end{lemma}
\begin{proof}
\begin{align*}
&\sqt{n-1} \sqt{k-1} - (\sqt{n+k-2} + \sqt{k-1,1} \sqt{n-2}) \\=& \sum_{i=0}^{\min(k-2,n-3)}\sqt{(n+k-i-4,i+1,1)}=0.
\hspace*{3cm}\qedhere
\end{align*}
\end{proof}

Since $\nabla a = a$ for any constant $a$, \lref{1} and \lref{newlemma} allow us to write a recursion for $\Qmn{m,3}$ operator that
\begin{equation}\label{algproof}
\Qmn{m+3,3}(1) =\nabla \Qmn{m,3}\nabla^{-1} (1)=\nabla \Qmn{m,3} (1).
\end{equation}
Using the recursion, we can prove \tref{1} by inducting on $m$. We shall give the complete algebraic proof of Equation (\ref{maineqn}) in \tref{1}, and omit the algebraic proof of Equations (\ref{maineqn2}) and (\ref{maineqn3}), only listing the base cases that
\begin{align}
\Qmn{2,3}(1)&=s_{21}+\sqt{1}s_{111},\\
\Qmn{3,3}(1)&=s_3+(\sqt{2}+2\sqt{1}+1)s_{21}+(\sqt{11}+\sqt{1}+\sqt{2}+\sqt{3}).
\end{align}

\bigskip
\begin{proof}[Proof of Equation (\ref{maineqn})]
When $k=0$, we can obtain by direct computation that
\begin{equation}
\Qmn{1,3}(1)=s_{111},
\end{equation}
which satisfies Equation (\ref{maineqn}). Then we induct on $k$ to prove Equation (\ref{maineqn}) that suppose the Schur function coefficients of $\Qmn{3k+1,3}(1)$ are the following:
\begin{align}
\scoeff{3}_{3k+1,3}&=\sum_{i=0}^{k-1}\sqt{(k+2i-1,k-i-1)},\\ 
\scoeff{21}_{3k+1,3}&=\sum_{i=0}^{k-1}\left(
\sqt{(k+2i,k-i-1)}+\sqt{(k+2i+1,k-i-1)}
\right),\\
\scoeff{111}_{3k+1,3}&=\sum_{i=0}^{k}\sqt{(k+2i,k-i)},
\end{align}
we want to show that 
\begin{align}
\label{alg1}\scoeff{3}_{3k+1,3}&=\sum_{i=0}^{k}\sqt{(k+2i-1,k-i-1)},\\ 
\label{alg2}\scoeff{21}_{3k+1,3}&=\sum_{i=0}^{k}\left(
\sqt{(k+2i,k-i-1)}+\sqt{(k+2i+1,k-i-1)}
\right),\\
\label{alg3}\scoeff{111}_{3k+1,3}&=\sum_{i=0}^{k+1}\sqt{(k+2i,k-i)},
\end{align}
One can directly compute that 
\begin{align}
\nabla s_3&= \sqt{2,2}s_{21}+\sqt{3,2}s_{111},\\
\nabla s_{21}&= \sqt{2,1}s_{21}-\sqt{3,1}s_{111},\\
\nabla s_{111}&= s_3+(\sqt{1}+\sqt{2})s_{21}+(\sqt{11}+\sqt{3})s_{111}.
\end{align}
By Equation (\ref{algproof}), we have
\begin{align*}
	\Qmn{3(k+1)+1,3}(1)=&\scoeff{3}_{3k+4,3}s_3+\scoeff{21}_{3k+4,3}s_{21}+\scoeff{111}_{3k+4,3}s_{111}\nonumber\\
	=&\nabla \Qmn{3k+1,3}(1)\nonumber\\
	=& \nabla (\scoeff{3}_{3k+1,3}s_3+\scoeff{21}_{3k+1,3}s_{21}+\scoeff{111}_{3k+1,3}s_{111})\nonumber\\
	=& \scoeff{3}_{3k+1,3}\nabla s_3+\scoeff{21}_{3k+1,3}\nabla s_{21}+\scoeff{111}_{3k+1,3}\nabla s_{111}\nonumber\\
	=& \scoeff{111}_{3k+1,3}s_3\nonumber\\
	&+\left(\sqt{2,2}\scoeff{3}_{3k+1,3}-\sqt{21}\scoeff{21}_{3k+1,3}+(\sqt{1}+\sqt{2})\scoeff{111}_{3k+1,3}\right)s_{21}\nonumber\\
	&+\left(\sqt{3,2}\scoeff{3}_{3k+1,3}-\sqt{31}\scoeff{21}_{3k+1,3}+(\sqt{11}+\sqt{3})\scoeff{111}_{3k+1,3}\right)s_{111},
\end{align*}
which implies that
\begin{align}
\scoeff{3}_{3k+4,3}&=\scoeff{111}_{3k+1,3},\\
\scoeff{21}_{3k+4,3}&=\sqt{2,2}\scoeff{3}_{3k+1,3}-\sqt{21}\scoeff{21}_{3k+1,3}+(\sqt{1}+\sqt{2})\scoeff{111}_{3k+1,3},\\ \scoeff{111}_{3k+4,3}&=\sqt{3,2}\scoeff{3}_{3k+1,3}-\sqt{31}\scoeff{21}_{3k+1,3}+(\sqt{11}+\sqt{3})\scoeff{111}_{3k+1,3}.
\end{align}
By the recursions above, one can verify Equations (\ref{alg1}), (\ref{alg2}) and (\ref{alg3}) using \lref{2}.
\end{proof}

\subsection{Combinatorial side --- $\Hikita_{m,3}[X; q, t]$}

Now we consider the Hikita polynomial defined by Equation (\ref{Hikitas}).
Any parking function $\PF\in\PFc_{m,3}$ has $3$ rows, thus  only  has $3$ cars: $\{1,2,3\}$, and the word $\sigma(\PF)$ can be any permutation $\sigma\in\mathcal{S}_3$. \taref{1} shows the $s_{\pides}$ contribution of the $6$ permutations in $\mathcal{S}_3$.
\begin{table}[ht!]
	\centering
	\caption{$s_{\pides}$ contribution of permutations in $\mathcal{S}_3$.}
	\begin{tabular}{|c|C{1.4cm}|C{1.4cm}|C{1.4cm}|C{1.4cm}|C{1.4cm}|C{1.4cm}|}
		\hline
		$\sigma\in\mathcal{S}_3$ & $123$ & $132$ & $213$ & $231$ & $312$ & $321$\\\hline
		$s_{\pides}$ & $s_{3}$ & $s_{21}$ & $s_{12}=0$ & $s_{21}$ & $s_{12}=0$ & $s_{111}$\\\hline
	\end{tabular}
	\label{table:1}
\end{table}

By our notation, $\Hikita_{m,3}[X; q, t]=[s_{3}]_{m,3}s_3+[s_{21}]_{m,3}s_{21}+[s_{111}]_{m,3}s_{111}$. We can work out the combinatorial side 
of the Rational Shuffle Theorem in the case where $n=3$ using (\ref{Hikitas}).

\subsubsection{Combinatorics of $\Hikita_{3k+1,3}[X; q, t]$}

We show the combinatorics of $\Hikita_{3k+1,3}[X; q, t]$ by enumerating the parking functions on $(3k+1)\times3$ lattice to prove the following formulas for the coefficients of Schur functions in $\Hikita_{3k+1,3}[X; q, t]$ (in $q,t$-analogue notation):
\begin{align}
\scoeff{3}_{3k+1,3}&=\sum_{i=0}^{k-1}(qt)^{k-1-i}\qtn{3i+1},\\ \scoeff{21}_{3k+1,3}&=\sum_{i=0}^{k-1}(qt)^{k-1-i}(\qtn{3i+2}+\qtn{3i+3}),\\
\scoeff{111}_{3k+1,3}&=\sum_{i=0}^{k}(qt)^{k-1-i}\qtn{3i+1}.
\end{align}

Given a parking function $\PF\in\PFc_{3k+1,3}$, we let $\Pi=\Pi(\PF)$ be the path of $\pi$. Its dinv correction is non-negative since $3k+1>3$ for $k\geq 1$, and
\begin{equation}
\dinvcorr(\PF) = \sum_{c\in\lambda(\Pi)}\chi\left( \frac{arm(c)+1}{leg(c)+1}\leq\frac{m}{n}<\frac{arm(c)}{leg(c)} \right).
\end{equation}

The partition corresponding to the Dyck path $\Pi$ has at most $2$ parts, so $\leg(c)$ of a cell $c\in\lambda(\Pi)$ is either $0$ or $1$. Taking \fref{10} for reference, we have 
\begin{enumerate}[(a)]
	\item $c\in\lambda(\Pi)$ with $\leg(c)=0$ and $1\leq\arm(c)<k$ contributes $1$ to dinv correction, marked $\bigcirc$ in \fref{10},
	\item $c\in\lambda(\Pi)$ with $\leg(c)=1$ and $k<\arm(c)\leq 2k-1$ contributes $1$ to dinv correction, marked $\bigtriangleup$ in \fref{10}.
\end{enumerate}

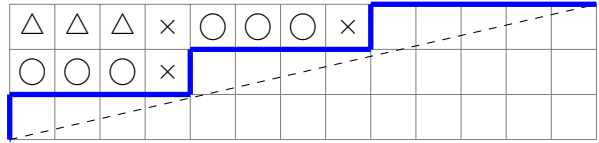
\begin{figure}[ht!]
	\centering
	\begin{tikzpicture}[scale=0.6]
	\Dpath{0,0}{13}{3}{0,4,8,-1};
	\fillcir{5/3,6/3,7/3,1/2,2/2,3/2};
	\filltri{1/3,2/3,3/3};
	\fillcro{4/2,4/3,8/3};
	\end{tikzpicture}
	\caption{The dinv correction of a $(3k+1,3)$-Dyck path when $k=4$.}
	\label{fig:10}
\end{figure}

Further, we can directly count the statistics area and dinv correction (dinvcorr) from the partition $\lambda(\Pi)$. 
We write $\lambda=(\lambda_1,\lambda_2)=\lambda(\Pi)$, then  $\lambda\subseteq\lambda_0=(2k,k)$, i.e.\ $\lambda_1\leq 2k$ and $\lambda_2\leq k$. Clearly, the $\area$ of $\Pi$ is counted by $|\lambda_0|-|\lambda|$, i.e.\ 
\begin{equation}\label{areaf}
\area(\Pi)=3k-\lambda_1-\lambda_2.
\end{equation}

We can also write the formula for dinv correction according to the partition $\lambda$:
\begin{equation}\label{dinvf}
\dinvcorr(\Pi) = \begin{cases} 
\lambda_1-1 & \textrm{if } \lambda_2=0 \textrm{ and } \lambda_1\leq k,\\
k-1 & \textrm{if } \lambda_2=0 \textrm{ and } \lambda_1> k,\\ 
\lambda_1-1 & \textrm{if } \lambda_2=\lambda_1\geq1,\\
\lambda_1-2 & \textrm{if } \lambda_2\geq1, 1\leq\lambda_1-\lambda_2\leq k,\textrm{ and } \lambda_1\leq k,\\
2\lambda_1-k-3 & \textrm{if } \lambda_2\geq1, 1\leq\lambda_1-\lambda_2\leq k,\textrm{ and } \lambda_1\geq k+1,\\
2\lambda_2+k-2 & \textrm{if } \lambda_2\geq1 \textrm{ and } \lambda_1-\lambda_2\geq k+1.\\
\end{cases}
\end{equation}

Note that the return statistic is always 1 since $3k+1$ and $3$ are coprime. We shall compute $[s_{3}]_{3k+1,3}$ first. 

From \taref{1}, we see that only the parking functions in $\PFc_{3k+1,3}$ with word $123$ contribute to the coefficient of $s_{3}$. We also notice that the $3$ cars should be in different columns, otherwise there are cars $i<j$ with $\rank(i)<\rank(j)$, contradicting with the restriction that the 
word of the parking function is $123$. Thus we have one $\PF\in\PFc_{3k+1,3}$ with word $123$ on each $(3k+1,3)$ Dyck path which has no consecutive north steps. 

Let $\lambda(\PF)=(\lambda_1,\lambda_2)$ be the partition associated to the Dyck path $\Pi(\PF)$ (see \fref{4}), then $\area(\PF)$ is counted by Equation (\ref{areaf}). Since the ranks of cars $1,2,3$ are decreasing, there is always no {tdinv}, thus $\dinv(\PF)=\dinvcorr(\Pi)$, which is counted by the latter 3 cases (since $\la_1>\la_2>0$) of Equation (\ref{dinvf}).
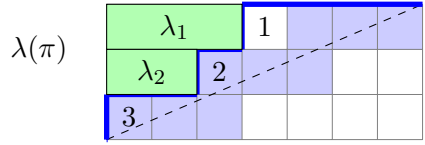
\begin{figure}[ht!]
	\centering	
	\begin{tikzpicture}[scale=0.6]
	\fillshade{1/1,2/1,3/1,3/2,4/2,5/2,5/3,6/3,7/3};
	\PFmnu{0,0}{7}{3}{0/3,2/2,3/1,0/0};
	\draw[fill=green!30!white] (0,2) rectangle (3,3);
	\draw[fill=green!30!white] (0,1) rectangle (2,2);
	\node at (1.5,2.5) {$\lambda_1$};
	\node at (1,1.5) {$\lambda_2$};
	\node at (-1.5,2) {$\lambda(\PF)$};
	\end{tikzpicture}
	\caption{Example: a parking function $\PF\in\PFc_{7,3}$ with word $123$.}
	\label{fig:4}
\end{figure}

For $[s_{3}]_{3k+1,3}=\sum_{i=0}^{k-1}(qt)^{k-1-i}\qtn{3i+1}$, we construct each term $(qt)^{k-1-i}\qtn{3i+1}$ as a sequence of parking functions. Since each parking function corresponds to a unique partition $\lambda\subset(2k,k)$ with 2 distinct parts, we shall use partitions to represent parking functions in $\PFc_{3k+1,3}$ with diagonal word $123$. For each $i$, we define the following $3$ branches of partitions (parking functions with word 123):
\begin{align*}
	\Lambda_1&=\{(k+i+1,k),(k+i,k-1),\ldots,(k+2,k-i+1)\},\\
	\Lambda_2&=\{(2k,i),(2k-1,i-1),\ldots,(2k+1-i,1)\},\\
	\Lambda_3&=\{(k+1,k-i),(k,i+1),\ldots,(k-i+1,k-i)\}.
\end{align*}
\begin{figure}[h]
	\centering	
	\resizebox{\columnwidth}{!}{
		\begin{tikzpicture}[scale=0.42]
		\partitiontwo{0}{0}{11}{k+i+1}{6}{k};
		\partitiontwo{0}{-4}{10}{k+i}{5}{k-1};
		\partitiontwo{0}{-8}{9}{k+i+1-r}{4}{k-r};
		\partitiontwo{0}{-12}{8}{k+2}{3}{k-i+1};
		\partitiontwo{0}{-16}{7}{k+1}{2}{k-i};
		
		\partitiontwo{13}{0}{12}{2k}{4}{i};
		\partitiontwo{13}{-4}{11}{2k-1}{3}{i-1};
		\partitiontwo{13}{-8}{10}{2k-r}{2}{i-r};
		\partitiontwo{13}{-12}{9}{2k+1-i}{1}{1};
		
		\partitiontwo{8}{-16}{6}{k}{2}{k-i};
		\partitiontwo{15}{-16}{5}{k-1}{2}{k-i};
		\partitiontwo{22}{-16}{3}{k-i+1}{2}{k-i};
		
		\arrowx{9.5,-1.5}{12.5,-1.5}
		\arrowx{9.5,-5.5}{12.5,-5.5}
		\arrowx{9.5,-9}{12.5,-9}
		\arrowx{9,-13}{12,-13}
		
		\arrowx{12.5,-2}{10,-3.5}
		\arrowx{12.5,-6}{11.5,-6.5}
		\arrowx{10,-7.5}{9,-8}
		\dotdot{11,-6.5}{11,-7}{11,-7.5};
		\arrowx{12.5,-9.5}{11,-10}
		\arrowx{9.5,-11}{8,-11.5}
		\dotdot{10.25,-10.5}{10.25,-11}{10.25,-11.5};
		\arrowx{12.5,-13.5}{7,-15.5}
		
		\arrowx{4.5,-17.5}{7.5,-17.5}
		\arrowx{11.5,-17.5}{14.5,-17.5}
		\arrowx{18,-17.5}{20,-17.5}
		\dotdot{20.5,-17.5}{21,-17.5}{21.5,-17.5}
		
		\fille{0}{-1}{q^{3i}\cdot(qt)^{k-1-i}}
		\fille{0}{-5}{q^{3i-2}t^2\cdot(qt)^{k-1-i}}
		\fille{0}{-9}{q^{3i-2r}t^{2r}\cdot(qt)^{k-1-i}}
		\fille{0}{-13}{q^{i+2}t^{2i-2}\cdot(qt)^{k-1-i}}
		\fille{0}{-17}{q^{i}t^{2i}\cdot(qt)^{k-1-i}}
		
		\fillw{25}{-1}{q^{3i-1}t\cdot(qt)^{k-1-i}}
		\fillw{24}{-5}{q^{3i-3}t^3\cdot(qt)^{k-1-i}}
		\fillw{23}{-9}{q^{3i-2r-1}t^{2r+1}\cdot(qt)^{k-1-i}}
		\fillw{22}{-13}{q^{i+1}t^{2i-1}\cdot(qt)^{k-1-i}}
		
		\fillls{12}{-15.5}{q^{i-1}t^{2i+1}\cdot(qt)^{k-1-i}}
		\fillls{19}{-15.5}{q^{i-2}t^{2i+2}\cdot(qt)^{k-1-i}}
		\fillw{25}{-17}{t^{3i}\cdot(qt)^{k-1-i}}
		\draw [thick, blue,decorate,decoration={brace,amplitude=5pt,mirror},xshift=0.4pt,yshift=-0.4pt](-7.5,0) -- (-7.5,-14) node[blue,midway,xshift=-.5cm] {$\Lambda_1$};
		\draw [thick, blue,decorate,decoration={brace,amplitude=5pt,mirror},xshift=0.4pt,yshift=-0.4pt](32,-14) -- (32,0) node[blue,midway,xshift=.5cm] {$\Lambda_2$};
		\draw [thick, blue,decorate,decoration={brace,amplitude=5pt,mirror},xshift=0.4pt,yshift=-0.4pt](-5.5,-18.5) -- (30,-18.5) node[blue,midway,yshift=-.5cm] {$\Lambda_3$};
		\end{tikzpicture}
	}
	\caption{The construction of $(qt)^{k-1-i}\qtn{3i+1}$.}
	\label{fig:5}
\end{figure}
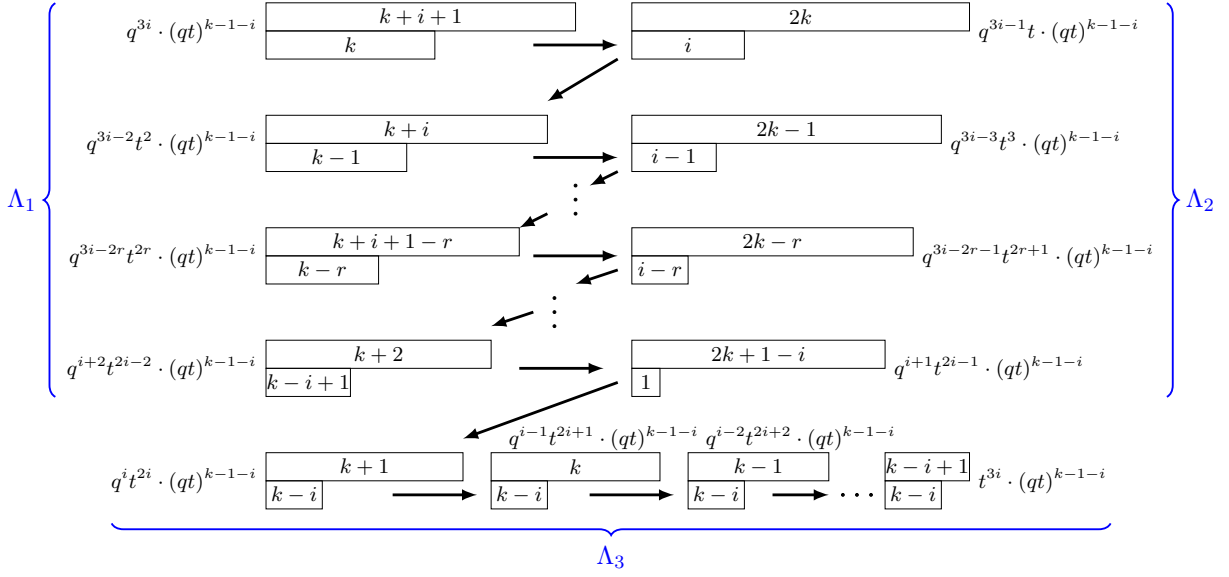

The branch $\Lambda_1$ contains all the partitions $\lambda$ such that $\lambda_1-\lambda_2=i+1\leq k$ with $\lambda_2> k-i$, 
the branch $\Lambda_2$ contains all the partitions $\lambda$ such that $\lambda_1-\lambda_2=2k-i>k$, and
the branch $\Lambda_3$ contains all the partitions $\lambda$ such that $\lambda_2=i+1$ and $\lambda_1-\lambda_2\leq k-i$. Notice that $|\Lambda_1|=|\Lambda_2|$. As shown in \fref{5}, the {construction} begins with \textit{alternatively} taking partitions from $\Lambda_1$ and $\Lambda_2$, {ending with} the {last partition of $\Lambda_2$}. Then {continue} the chain by taking partitions in $\Lambda_3$ and end the chain with the last partition $(k-i+1,k-i)$ in $\Lambda_3$. The weights of the parking functions are $(qt)^{k-1-i}q^{3i},(qt)^{k-1-i}q^{3i-1}t,\ldots,(qt)^{k-1-i}t^{3i}$ following the order of the chain.

To be more precise, it is not difficult to check that each parking function with diagonal word 123 is contained in $\Lambda_1\cup \Lambda_2\cup\Lambda_3$ for some $i$, and the parking function weights are
\begin{align}
\sum_{\PF\in\Lambda_1}t^{\area(\PF)}q^{\dinv(\PF)}&=(qt)^{k-i-1}q^{i+2}\qtnn{i},\\
\sum_{\PF\in\Lambda_2}t^{\area(\PF)}q^{\dinv(\PF)}&=(qt)^{k-i-1}q^{i+1}t\qtnn{i},\quad\mbox{and}\\
\sum_{\PF\in\Lambda_3}t^{\area(\PF)}q^{\dinv(\PF)}&=(qt)^{k-i-1}t^{2i}\qtn{i+1},
\end{align}
which sum up to $(qt)^{k-1-i}\qtn{3i+1}$. This proves that $[s_{3}]_{3k+1,3}=\sum_{i=0}^{k-1}(qt)^{k-1-i}\qtn{3i+1}$.
\fref{6} shows an example of the combinatorial construction of the coefficient $[s_{3}]_{10,3}$.

\begin{figure}[ht!]
	\centering	
	\resizebox{\columnwidth}{!}{\begin{tikzpicture}[scale=.25]
		\PFmnum{0,0}{10}{3}{0/3,3/2,6/1,0/0};
		\PFmnum{0,-4}{10}{3}{0/3,2/2,5/1,0/0};
		\PFmnum{0,-8}{10}{3}{0/3,1/2,4/1,0/0};
		\PFmnum{0,-12}{10}{3}{0/3,1/2,2/1,0/0};
		
		\PFmnum{11,0}{10}{3}{0/3,2/1,6/2,0/0};
		\PFmnum{11,-4}{10}{3}{0/3,1/1,5/2,0/0};
		\PFmnum{11,-8}{10}{3}{0/3,1/2,3/1,0/0};
		\draw [thick, decorate,decoration={brace,amplitude=5pt,mirror},xshift=0.4pt,yshift=-0.4pt](-.5,3) -- (-.5,-12) node[midway,xshift=-.6cm] {$\qtn{7}$};
		\path (23,0);
		\end{tikzpicture}
		\begin{tikzpicture}[scale=.25]
		\PFmnum{0,0}{10}{3}{0/3,3/2,5/1,0/0};
		\PFmnum{0,-4}{10}{3}{0/3,2/2,4/1,0/0};
		\PFmnum{11,0}{10}{3}{0/3,1/1,6/2,0/0};
		\PFmnum{11,-4}{10}{3}{0/3,2/2,3/1,0/0};
		\draw [thick, decorate,decoration={brace,amplitude=5pt,mirror},xshift=0.4pt,yshift=-0.4pt](21.5,-4) -- (21.5,3) node[midway,xshift=.9cm] {$(qt)\qtn{4}$};
		
		\PFmnum{0,-10}{10}{3}{0/3,3/2,4/1,0/0};
		\draw [thick, decorate,decoration={brace,amplitude=5pt,mirror},xshift=0.4pt,yshift=-0.4pt](10.5,-10) -- (10.5,-7) node[midway,xshift=1cm] {$(qt)^2\qtn{1}$};
		\path (0,-11);
		\end{tikzpicture}}
	\caption{The construction of $[s_{3}]_{10,3}=\qtn{7}+(qt)\qtn{4}+(qt)^2\qtn{1}$.}
	\label{fig:6}	
\end{figure}
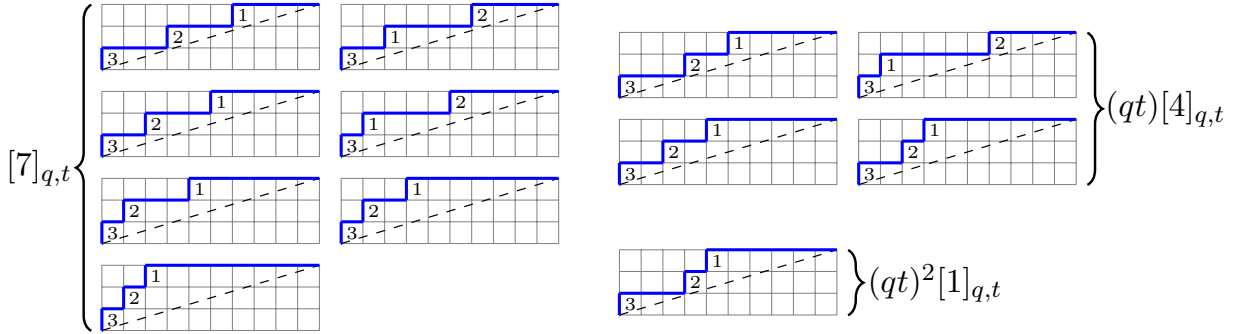

We can combinatorially prove  $[s_{21}]_{3k+1,3}=\sum_{i=0}^{k-1}(qt)^{k-1-i}(\qtn{3i+2}+\qtn{3i+3})$ in a similar way. In this case, we have $2$ possible diagonal words: $132$ and $312$. In both cases, the car 2 has the smallest rank, which means the label of the first (lowest) row must be 2. Thus, the pair of cars $(1,2)$ does not produce a tdinv. If we let $\lambda=(\lambda_1,\lambda_2)$ be the partition corresponding to the path $\Pi$, and let the labels of row 1, row 2, row 3 (counting from bottom to top) be $\ell_1,\ell_2,\ell_3$, then we have the following formula for temporary dinv:
\begin{equation}\label{tdinv}
\tdinv(\PF) = \begin{cases} 
\chi(\ell_3>\ell_2) & \textrm{if } \lambda_2=0 \textrm{ and } \lambda_1\leq k,\\
\chi(\ell_3>\ell_1)+\chi(\ell_2>\ell_3) & \textrm{if } \lambda_2=0 \textrm{ and } \lambda_1> k,\\ 
\chi(\ell_2>\ell_1) & \textrm{if } \lambda_2=\lambda_1\geq1,\\
\chi(\ell_2>\ell_1)+\chi(\ell_3>\ell_2) & \textrm{if } \lambda_2\geq1, 1\leq\lambda_1-\lambda_2\leq k,\ \lambda_1\leq k,\\
\chi(\ell_2>\ell_1)+\chi(\ell_3>\ell_1)+\chi(\ell_3>\ell_2) & \textrm{if } \lambda_2\geq1, 1\leq\lambda_1-\lambda_2\leq k,\ \lambda_1\geq k+1,\\
\chi(\ell_2>\ell_1)+\chi(\ell_3>\ell_1)+\chi(\ell_2>\ell_3) & \textrm{if } \lambda_2\geq1 \textrm{ and } \lambda_1-\lambda_2\geq k+1.\\
\end{cases}
\end{equation}

In the construction of the coefficient $\scoeff{21}_{3k+1,3} = \sum_{i=0}^{k-1}(qt)^{k-1-i} (\qtn{3i+2}+\qtn{3i+3})$, we construct each term $(qt)^{k-1-i}\qtn{3i+2}$ or $(qt)^{k-1-i}\qtn{3i+3}$ as a sequence of parking functions. First, we define the following 3 branches of parking functions to obtain the term $(qt)^{k-1-i}\qtn{3i+3}$:
\begin{align*}
	\Lambda_1&=\{\PF:\lambda(\PF)\in\{(2k,i+1),(2k-1,i),\ldots,(2k-i,1)\}, (\ell_1,\ell_2,\ell_3)=(2,1,3)\},\\
	\Lambda_2&=\{\PF:\lambda(\PF)\in\{(2k,i),(2k-1,i-1),\ldots,(2k-i,0)\}, (\ell_1,\ell_2,\ell_3)=(2,3,1)\},\\
	\Lambda_3&=\{\PF:\lambda(\PF)\in\{(k,k-i-1),\ldots,(k-i,k-i-1)\}, (\ell_1,\ell_2,\ell_3)=(2,3,1)\}.
\end{align*}

With the 3 branches defined, the {construction} is similar to that of $(qt)^{k-1-i}\qtn{3i+1}$ as a component of $[s_{3}]_{3k+1,3}$. We \textit{alternatively} take parking functions from $\Lambda_1$ and $\Lambda_2$, {ending with} the {last partition of $\Lambda_2$}. Then we {continue} the chain by taking partitions in $\Lambda_3$ ending with the last parking function corresponding to the partition $(k-i,k-i-1)$ with labels $(\ell_1,\ell_2,\ell_3)=(2,3,1)$ in $\Lambda_3$. The weights of the parking functions are $(qt)^{k-1-i}q^{3i+2},\ldots,(qt)^{k-1-i}t^{3i+2}$, which sum up to $(qt)^{k-1-i}\qtn{3i+3}$.

Second, we define another three branches of parking functions for $(qt)^{k-1-i}\qtn{3i+2}$:
\begin{align*}
	\Lambda_4&=\{\PF:\lambda(\PF)\in\{(k+i+1,k),(k+i,k-1),\ldots,(k+1,k-i)\}, (\ell_1,\ell_2,\ell_3)=(2,3,1)\},\\
	\Lambda_5&=\{\PF:\lambda(\PF)\in\{(k+i,k),(k+i-1,k-1),\ldots,(k,k-i)\}, (\ell_1,\ell_2,\ell_3)=(2,1,3)\},\\
	\Lambda_6&=\{\PF:\lambda(\PF)\in\{(k-1,k-i),\ldots,(k-i,k-i)\}, (\ell_1,\ell_2,\ell_3)=(2,1,3)\}.
\end{align*}

The construction is the same as that of $(qt)^{k-1-i}\qtn{3i+3}$, and the weights of the parking functions are $(qt)^{k-1-i}q^{3i+1},\ldots,(qt)^{k-1-i}t^{3i+1}$ which sum up to $(qt)^{k-1-i}\qtn{3i+2}$.

Thus we have proved that  $\scoeff{21}_{3k+1,3} = \sum_{i=0}^{k-1}(qt)^{k-1-i} (\qtn{3i+2}+\qtn{3i+3})$. \fref{firsts21} shows an example of the combinatorial construction of the coefficient $[s_{21}]_{7,3}$.
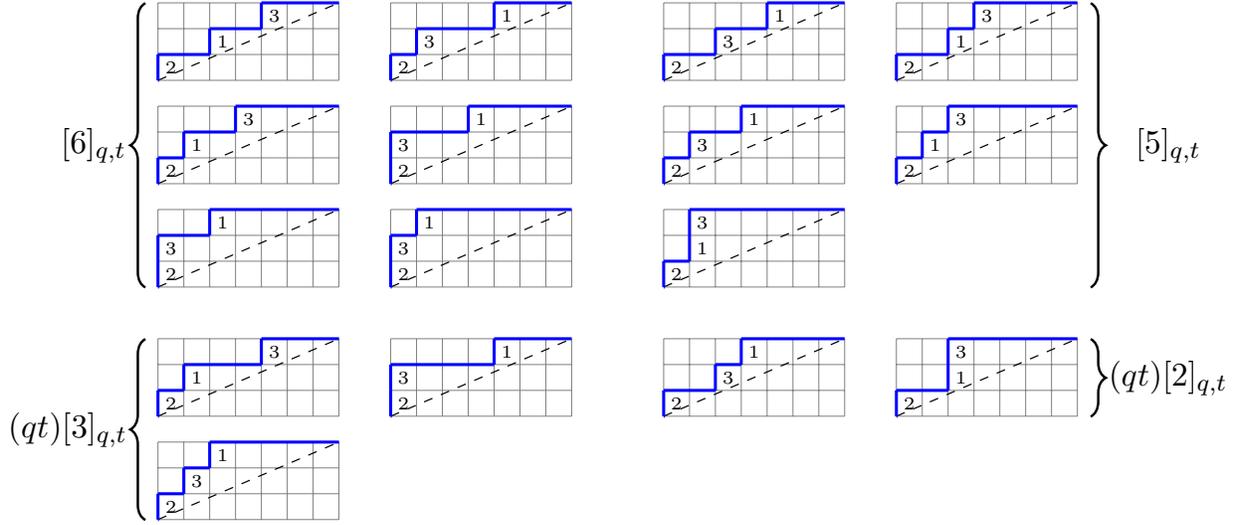
\begin{figure}[ht!]
	\centering
	\resizebox{\columnwidth}{!}{\begin{tikzpicture}[scale=.3]
		\PFmnum{0,0}{7}{3}{0/2,2/1,4/3,0/0};
		\PFmnum{0,-4}{7}{3}{0/2,1/1,3/3,0/0};
		\PFmnum{0,-8}{7}{3}{0/2,0/3,2/1,0/0};
		
		\PFmnum{9,0}{7}{3}{0/2,1/3,4/1,0/0};
		\PFmnum{9,-4}{7}{3}{0/2,0/3,3/1,0/0};
		\PFmnum{9,-8}{7}{3}{0/2,0/3,1/1,0/0};
		\draw [thick, decorate,decoration={brace,amplitude=5pt,mirror},xshift=0.4pt,yshift=-0.4pt](-.5,3) -- (-.5,-8) node[midway,xshift=-.6cm] {$\qtn{6}$};
		
		\PFmnum{0,-13}{7}{3}{0/2,1/1,4/3,0/0};
		\PFmnum{0,-17}{7}{3}{0/2,1/3,2/1,0/0};
		\PFmnum{9,-13}{7}{3}{0/2,0/3,4/1,0/0};
		\draw [thick, decorate,decoration={brace,amplitude=5pt,mirror},xshift=0.4pt,yshift=-0.4pt](-.5,-10) -- (-.5,-17) node[midway,xshift=-.9cm] {$(qt)\qtn{3}$};
		
		\path (19,-18);
		\end{tikzpicture}
		\begin{tikzpicture}[scale=.3]
		\PFmnum{0,0}{7}{3}{0/2,2/3,4/1,0/0};
		\PFmnum{0,-4}{7}{3}{0/2,1/3,3/1,0/0};
		\PFmnum{0,-8}{7}{3}{0/2,1/1,1/3,0/0};
		
		\PFmnum{9,0}{7}{3}{0/2,2/1,3/3,0/0};
		\PFmnum{9,-4}{7}{3}{0/2,1/1,2/3,0/0};
		\draw [thick, decorate,decoration={brace,amplitude=5pt,mirror},xshift=0.4pt,yshift=-0.4pt](16.5,-8) -- (16.5,3) node[midway,xshift=.9cm] {$\qtn{5}$};
		
		\PFmnum{0,-13}{7}{3}{0/2,2/3,3/1,0/0};	
		\PFmnum{9,-13}{7}{3}{0/2,2/1,2/3,0/0};
		\draw [thick, decorate,decoration={brace,amplitude=5pt,mirror},xshift=0.4pt,yshift=-0.4pt](16.5,-13) -- (16.5,-10) node[midway,xshift=.9cm] {$(qt)\qtn{2}$};
		\path (0,-18);
		\end{tikzpicture}}
	\caption{The construction of $[s_{21}]_{7,3}=\qtn{6}+\qtn{5}+(qt)(\qtn{3}+\qtn{2})$.}
	\label{fig:firsts21}	
\end{figure}

The identity that $[s_{111}]_{3k+1,3}=\sum_{i=0}^{k}(qt)^{k-i}\qtn{3i+1}=[s_3]_{3k+4,3}$ is a consequence of the following corollary of \tref{3} (a):
\begin{corollary}\label{corm3}
	For any $m>0$, $[s_{111}]_{m,3}=[s_3]_{m+3,3}$. 
\end{corollary}

\subsubsection{Combinatorics of $\Hikita_{3k+2,3}[X; q, t]$}

We study the combinatorics of $\Hikita_{3k+2,3}[X; q, t]$ in a similar manner by enumerating the parking functions on the $(3k+2)\times3$ lattice to prove the following formulas for the coefficients of Schur functions in $\Hikita_{3k+2,3}[X; q, t]$ (in $q,t$-analogue notation):
\begin{align}
\scoeff{3}_{3k+2,3}&=\sum_{i=0}^{k-1}(qt)^{k-1-i}\qtn{3i+2},\\ \scoeff{21}_{3k+2,3}&=\sum_{i=-1}^{k-1}(qt)^{k-1-i}(\qtn{3i+3}+\qtn{3i+4}), \mbox{ \ \ and}\\
\scoeff{111}_{3k+2,3}&=\sum_{i=0}^{k}(qt)^{k-i}\qtn{3i+2}.
\end{align}

Given a parking function $\PF\in\PFc_{3k+2,3}$ with $\Pi(\PF)=\Pi$, we can compute the dinv correction of $\pi$ by examining the  cells $c\in\lambda(\Pi)$. Taking \fref{divcorr3k2} for reference,
\begin{enumerate}[(a)]
	\item $c\in\lambda(\Pi)$ with $\leg(c)=0$ and $1\leq\arm(c)< k$ contributes $1$ to dinv correction, marked $\bigcirc$ in \fref{divcorr3k2},
	\item $c\in\lambda(\Pi)$ with $\leg(c)=1$ and $k<\arm(c)\leq 2k$ contributes $1$ to dinv correction, marked $\bigtriangleup$ in \fref{divcorr3k2}.
\end{enumerate}

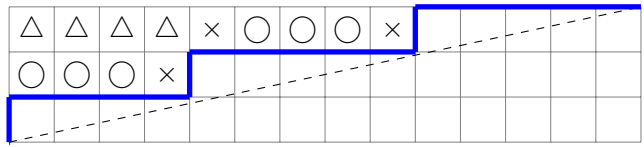
\begin{figure}[ht!]
	\centering
	\begin{tikzpicture}[scale=0.6]
	\Dpath{0,0}{14}{3}{0,4,9,-1};
	\fillcir{6/3,7/3,8/3,1/2,2/2,3/2};
	\filltri{1/3,2/3,3/3,4/3};
	\fillcro{4/2,5/3,9/3};
	\end{tikzpicture}
	\caption{The dinv correction of a $(3k+2,3)$-Dyck path when $k=4$.}
	\label{fig:divcorr3k2}
\end{figure}

Further, we can directly count the statistics area and dinv correction (dinvcorr) from the partition $\lambda(\Pi)=(\la_1,\la_2)\subseteq (2k+1,k)$ of a $(3k+2,3)$-Dyck path $\Pi$. Similar to Equation (\ref{areaf}),  we have
\begin{equation}\label{areaff}
\area(\Pi)=3k+1-\lambda_1-\lambda_2.
\end{equation}
The dinv correction formula is the same as Equation (\ref{dinvf}), and the return is still always 1. 

To prove $\scoeff{3}_{3k+2,3}=\sum_{i=0}^{k-1}(qt)^{k-1-i}\qtn{3i+2}$, we shall construct the following $3$ branches of partitions (parking functions with word 123) for each term $(qt)^{k-1-i}\qtn{3i+2}$:
\begin{align*}
	\Lambda_1&=\{(2k+1,i+1),(2k,i),\ldots,(2k+1-i,1)\},\\
	\Lambda_2&=\{(k+i+1,k),(k+i,k-1),\ldots,(k+1,k-i)\},\\
	\Lambda_3&=\{(k,k-i),(k-1,i+1),\ldots,(k-i+1,k-i)\}.
\end{align*}
Then, we can follow the same construction as the $(3k+1,3)$ case to obtain all parking functions with word $123$ and their weights $(qt)^{k-1-i}q^{3i+1},\ldots,(qt)^{k-1-i}t^{3i+1}$.

Similarly, to prove $\scoeff{21}_{3k+2,3}=\sum_{i=-1}^{k-1}(qt)^{k-1-i}(\qtn{3i+3}+\qtn{3i+4})$, we have $6$ branches of parking functions as follows:
\begin{align*}
	\Lambda_1&=\{\PF:\lambda(\PF)\in\{(k+i+2,k),\ldots,(k+2,k-i)\}, (\ell_1,\ell_2,\ell_3)=(2,3,1)\},\\
	\Lambda_2&=\{\PF:\lambda(\PF)\in\{(k+i+1,k),\ldots,(k+1,k-i)\}, (\ell_1,\ell_2,\ell_3)=(2,1,3)\},\\
	\Lambda_3&=\{\PF:\lambda(\PF)\in\{(k+1,k-i-1),\ldots,(k-i,k-i-1)\}, (\ell_1,\ell_2,\ell_3)=(2,3,1)\},\\
	\Lambda_4&=\{\PF:\lambda(\PF)\in\{(2k+1,i+1),\ldots,(2k-i+1,1)\}, (\ell_1,\ell_2,\ell_3)=(2,1,3)\},\\
	\Lambda_5&=\{\PF:\lambda(\PF)\in\{(2k+1,i),\ldots,(2k-i+1,0)\}, (\ell_1,\ell_2,\ell_3)=(2,3,1)\},\\
	\Lambda_6&=\{\PF:\lambda(\PF)\in\{(k,k-i),\ldots,(k-i,k-i)\}, (\ell_1,\ell_2,\ell_3)=(2,1,3)\}.
\end{align*}
Then, the total weight of parking functions in the first $3$ branches is $(qt)^{k-1-i}\qtn{3i+4}$, and the total weight of parking functions in the last $3$ branches is $(qt)^{k-1-i}\qtn{3i+3}$.

The proof of $\scoeff{111}_{3k+2,3}=\scoeff{3}_{3(k+1)+2,3}=\sum_{i=0}^{k}(qt)^{k-i}\qtn{3i+2}$ follows from Corollary \ref{corm3}.

\subsubsection{Combinatorics of $\Hikita_{3k,3}[X; q, t]$}

Notice that the area and dinv of parking functions in $\PFc_{3k,3}$ are equal to those of the parking functions in $\PFc_{3k+1,3}$. Given a parking function $\PF\in\PFc_{3k+1,3}$ where $\lambda(\PF)=\lambda=(\lambda_1,\lambda_2)$, the return statistic of $\PF$ is formulated as
\begin{equation*}
\ret(\PF)=2\chi(\lambda_1=2k)+\chi(\lambda_2=k)-2\chi(\lambda_1=2k)\cdot\chi(\lambda_2=k).
\end{equation*}
By the Extended Rational Shuffle Theorem in the non-coprime case, 
\begin{align}
\Hikita_{3k,3}[X; q, t]&=\sum_{\PF\in\PFc_{3k,3}} [\ret(\PF)]_{\frac{1}{t}}t^{\area(\PF)}q^{\dinv(\PF)}F_{\pides(\PF)}[X]\nonumber\\
&=\sum_{\PF\in\PFc_{3k+1,3}} [\ret(\PF)]_{\frac{1}{t}}t^{\area(\PF)}q^{\dinv(\PF)}s_{\pides(\PF)}.
\end{align}

To prove the following formulas for the coefficients of Schur functions in $\Hikita_{3k+1,3}[X; q, t]$,
\begin{align}
\label{3k1}\scoeff{3}_{3k,3}=&\sum_{i=0}^{k-1}(qt)^{k-1-i}(\qtn{3i-1}+\qtn{3i}+\qtn{3i+1}),\\ 
\label{3k2}\scoeff{21}_{3k,3}=&(qt)^{k-1}(\qtn{3}+2\qtn{2}+\qtn{1})\nonumber\\
&+\sum_{i=1}^{k-1}(qt)^{k-1-i}(\qtn{3i}+2\qtn{3i+1}+2\qtn{3i+2}+\qtn{3i+3}),\\
\label{3k3}\scoeff{111}_{3k,3}=&\sum_{i=0}^{k}(qt)^{k-i}(\qtn{3i-1}+\qtn{3i}+\qtn{3i+1}),
\end{align}
we use the constructions of $\scoeff{3}_{3k+1,3},\scoeff{21}_{3k+1,3},\scoeff{111}_{3k+1,3}$ and modify the weight of parking functions with nonzero returns. We use the same sets of partitions $\Lambda_1,\Lambda_2,\Lambda_3,\Lambda_4,\Lambda_5,\Lambda_6$ as Section 4.2.1.

For $\scoeff{3}_{3k,3}$, the first parking function in each set $\Lambda_1$ and $\Lambda_2$ has return statistic 1, except that the first parking function in $\Lambda_1$ when $i=k-1$ has return 2. All the remaining parking functions have return 3. Then we prove Equation (\ref{3k1}) by summing up the parking function weights.

For $\scoeff{21}_{3k,3}$, the first parking function in each of the sets $\Lambda_1,\Lambda_2,\Lambda_4,\Lambda_5$ has return statistic 1 since the second parts of these partitions are $k$, except that the first parking function in $\Lambda_1$ or $\Lambda_4$ when $i=k-1$ has return 2. All the remaining parking functions have return 3. Then again we obtain Equation (\ref{3k2}) by direct computation.

The proof of $\scoeff{111}_{3k,3}=\scoeff{3}_{3(k+1),3}$ follows from Corollary \ref{corm3}.

\section{Combinatorial results about Schur function expansions of the $(3,n)$ case}
\subsection{Recursive formula for $\scoeff{\lambda}_{3,n}$}

In $(3,n)$ case, we have $n$ cars, i.e.\ the word of a $(3,n)$ parking function is a permutation of $[n]$. By Remark \ref{remark:2}, $[s_{\lambda}]_{3,n} \neq 0$ implies that 
$\lambda$ must be of the form $3^a2^b1^c$ with $3a+2b+c=n$, i.e.\ $[s_{\lambda}]_{3,n} \neq 0$ only if the partition $\lambda$ only has parts of sizes less than or equal to $3$. 

We have the following corollary of \tref{3} summarizing several symmetries about $[s_{\lambda}]_{3,n}$.

\begin{corollary}\label{theorem:2} For all $n > 0$ and $a,b,c\geq 0$,
	\begin{enumerate}[(a)]
		\item $[s_{3^{a}2^b1^c}]_{3,n}=[s_{2^b1^c}]_{3,n-3a}$,
		\item $[s_{1^n}]_{3,n}=[s_{111}]_{n,3}$,
		\item $[s_{21^{n-2}}]_{3,n}=[s_{21}]_{n,3}$.
	\end{enumerate}
\end{corollary}

Further, 
we have found the straightening action in parking functions combinatorially from parking functions with $\pides$ $\{\cdots, 1, 3, \cdots\}$ to parking functions with  $\pides$ $\{\cdots, 2,2, \cdots\}$, which is an \textit{involution} $\Phi$ whose fixed points are the coefficients of $[s_{2^a1^b}]_{3,n}$. 
We call the fixed points of $\Phi$ the \emph{fixed parking functions}.
The details of the involution $\Phi$  will be given in Section 5.2.1. 

Let $a,b$ be positive integers. We have conjectured a bijection $\swi$ between the fixed parking functions with $\pides\ 2^a1^b$ and the fixed parking functions with $\pides\ 2^b1^a$ in the coprime case, mapping the $2$ cars (or $1$ car) causing part $2$ (or $1$) in $\pides\ 2^a1^b$ to $1$ car (or $2$ cars) causing part $1$ (or $2$) in $\pides\ 2^b1^a$. The map $\swi$ has nice properties summarized in \tref{areasw}, \tref{areasw1} and \cref{2a1bn} in Section 5.2.2, and they imply the coprime case of another important symmetry:
\begin{conjecture}\label{conjecture:2a1b}For all $a,b,n \geq 0$,
	$$[s_{2^a1^b}]_{3,n}=[s_{2^b1^a}]_{3,3(a+b)-n}.
	$$
\end{conjecture}

The results above 
show that the problem of computing 
the Schur function expansion of $\Qmn{3,n}(1)$ can be reduced to the problem of 
finding the coefficients of Schur functions of the form $s_{2^a1^b}$ where $a<b$. 
Finally, we conjecture a recursive formula for such coefficients $[s_{2^a1^b}]_{3,n}$ where $a<b$.
\begin{conjecture}
	Let $a<b$, then
	$$[s_{2^a1^b}]_{3,n}=(qt)[s_{2^a1^{b-3}}]_{3,n-3} + \sum_{i=0}^{a}\qtn{b+i}.$$
\end{conjecture}
We have verified this formula by Maple for $n<27$. If the conjectures are true, then we have solved the Schur function expansion in the $(3,n)$ case.

\subsection{The involution $\Phi$ and the map $\swi$}
We shall first introduce an involution on $(3,n)$-parking functions whose pideses contain $1,3$ or $2,2$. The fixed points of the involution is a subset of parking functions whose pideses do not contain $1,3$. Then, we conjecture a bijection between the fixed parking functions with $\pides\ 2^a1^b$ and the fixed parking functions with $\pides\ 2^b1^a$ for the coprime case.

\subsubsection{The involution $\Phi$}
An \emph{involution} $f$ of a set $S$ is a bijection from $S$ to itself, such that $f^2 = \mathrm{id}$ is the identity map. An element $s\in S$ such that $f(s)=s$ is called a \emph{fixed point} of the involution. 

Suppose that there is a weight function $w(s)$ for each elements $s\in S$.
A \emph{sign-reversing involution} $f$ of the set $S$ (with respect to the weight $w$) is an involution such that for all $s\in S$, if $f(s)\neq s$, then $w(f(s))=-w(s)$. As a consequence, we have
\begin{equation}
\sum_{s\in S}w(s) = \sum_{s\in S} w(f(s)) = \sum_{s\in S, f(s)=s} w(s),
\end{equation}
i.e.\ we only need to consider the fixed points of $f$ when computing the total weight of the set $S$.

By definition, 
\begin{equation}\label{PF3n}
\Qmn{3,n}(1)=\Hikita_{3,n}[X; q, t]=\sum_{\PF\in\PFc_{3,n}}
[\ret(\PF)]_{\frac{1}{t}} t^{\area(\PF)}q^{\dinv(\PF)}s_{\pides(\PF)}
\end{equation}
is a sum of weights of parking functions in $\PFc_{3,n}$, where the weight of a parking function $\PF$ is $[\ret(\PF)]_{\frac{1}{t}} t^{\area(\PF)}q^{\dinv(\PF)}s_{\pides(\PF)}$. In this section, we build a sign-reversing involution $\Phi$ for the set $\PFc_{3,n}$ in order to simplify the computation of the polynomial $\Qmn{3,n}(1)$.

Note that by the straightening action on Schur functions, we have 
\begin{equation}
s_{\lambda 13 \mu}=-s_{\lambda 22 \mu},
\end{equation}
where $\la$ and $\mu$ are two compositions, and $\lambda 13 \mu$ (or $\lambda 22 \mu$) is the composition obtained by first listing all the parts in $\la$, then two parts of sizes $1$ and $3$ (or 2 and 2), finally all the parts in $\mu$.

Let $\PFaa$ be the set of parking functions in $\PFc_{3,n}$ with pides $\lambda 13 \mu$ 
and $\PFbb$ be the set of parking functions in $\PFc_{3,n}$ with pides $\lambda 22 \mu$, 
then we can give an involution $\Phi$ of the set $\PFaa\cup\PFbb$, such that 
\begin{itemize}
	\item all the fixed points of $\Phi$ are in the set $\PFbb$, and
	\item the set of non-fixed points in $\PFbb$ are in bijection with the set $\PFaa$.
\end{itemize}

Let $\pi$ be a parking function in $\PFc_{3,n}$. If $\pides(\PF) = \lambda 13 \mu$, then without loss of generality, we suppose that the cars causing pides $13$ are $1,2,3,4$, which means that $\rank(1)<\rank(2)>\rank(3)>\rank(4)$. There are 3 possible subwords (subsequences of the words of $\pi$) formed by the $4$ cars, which are
$$
2341,\ \ 2314,\ \ 2134.
$$

On the other hand, the cars $2,3,4$ are in different columns since $\rank(2)>\rank(3)>\rank(4)$. Given that there are only 3 columns,  we have three possible placements of the four cars:
\begin{enumerate}[(I)]
	\item Cars $1$ and $4$ are in the same column.
	\item Cars $1$ and $2$ are in the same column.
	\item Cars $1$ and $3$ are in the same column.
\end{enumerate}

If the four cars form a word $2341$, then (I), (II), (III) are all possible; if the four cars form a word $2314$, then only (II), (III) are possible; if the four cars form a word $2134$, then only (II) is possible.

\bigskip
Next, we consider the case when $\pides(\PF) = \lambda 13 \mu$, i.e.\ the cars $1,2,3,4$ cause pides $22$, and $\rank(1)>\rank(2)<\rank(3)>\rank(4)$. The possible words are
$$
3412,\ 3142,\ 3124,\ 1324,\ 1342.
$$
The cars $1,2$ and the cars $3,4$ have to be in different columns since $\rank(1)>\rank(2)$ and $\rank(3)>\rank(4)$, thus we have the following five possible placements of the four cars:
\begin{enumerate}[(i)]
	\item Both cars $1,3$ and $2,4$ are in the same column.
	\item Only cars $1$ and $4$ are in the same column.
	\item Only cars $2$ and $4$ are in the same column.
	\item Only cars $1$ and $3$ are in the same column.
	\item Only cars $2$ and $3$ are in the same column.
\end{enumerate}
If the four cars form a word $3412$, then (i), (ii), (iii), (iv) and (v) are all possible; if the four cars form a word $3142$, then only (i), (iii), (iv) and (v) are possible; if the four cars form a word $3124$, then only (iv) and (v) are possible; if the four cars form a word $1324$, then only (v) is possible; if the four cars form a word $1342$, then only (iii) and (v) are possible. 

For any permutation $\sg\in\Sn{n}$ and any $\PF\in\PFc_{3,n}$, we let $\sg\cdot\PF$ be the parking function obtained by permuting the cars of $\PF$ by the permutation $\sg$. We also let $\word(\PF)$ be the word of the cars $1,2,3,4$. Then we can define the map $\Phi$ on the set  $\PFaa$ that
\begin{equation*}
\Phi|_{\PFaa}:\quad\PFaa \quad\rightarrow\quad\PFbb.
\end{equation*}
Following is the detailed definition, while the words and the placements of the images are recorded in each case:
\begin{equation}\label{Phimap}
\Phi(\PF) = \begin{cases} 
(1,2)\PF & \mbox{if } \word(\PF)=2341 \mbox{ and placement is (I). \ \ } \Phi(\PF) \mbox{ has word }1342 \mbox{ (iii).}\\
(1,2,3)\PF & \mbox{if } \word(\PF)=2341 \mbox{ and placement is (II). \ }\Phi(\PF) \mbox{ has word }3142 \mbox{ (v).}\\
(1,2)\PF & \mbox{if } \word(\PF)=2341 \mbox{ and placement is (III). }\Phi(\PF) \mbox{ has word }1342 \mbox{ (v).}\\
(1,2,3)\PF & \mbox{if } \word(\PF)=2314 \mbox{ and placement is (II). \ }\Phi(\PF) \mbox{ has word }3124 \mbox{ (v).}\\
(1,2)\PF & \mbox{if } \word(\PF)=2314 \mbox{ and placement is (III). }\Phi(\PF) \mbox{ has word }1324 \mbox{ (v).}\\
(2,3)\PF & \mbox{if } \word(\PF)=2134 \mbox{ and placement is (II). \ }\Phi(\PF) \mbox{ has word }3124 \mbox{ (iv).}
\end{cases}
\end{equation}
Then we shall define the map $\Phi$ on the set  $\PFbb$ that
\begin{equation*}
\Phi|_{\PFbb}:\quad\PFbb \quad\rightarrow\quad\PFaa\cup\PFbb.
\end{equation*}
Notice that the non-fixed points in $\PFbb$ are mapped into the set $\PFaa$. We have
\begin{equation}\label{Psimap}
\Phi(\PF) = \begin{cases} 
\PF & \mbox{if } \word(\PF)=3412 \mbox{, or } \word(\PF)=3142\mbox{ and placement is (i),(iii),(iv).}\\
(1,3,2)\PF & \mbox{if } \word(\PF)=3142 \mbox{ and placement is (v). \ } \Phi(\PF) \mbox{ has word }2341 \mbox{ (II).}\\
(2,3)\PF & \mbox{if } \word(\PF)=3124 \mbox{ and placement is (iv). } \Phi(\PF) \mbox{ has word }2134 \mbox{ (II).}\\
(1,3,2)\PF & \mbox{if } \word(\PF)=3124 \mbox{ and placement is (v). \ } \Phi(\PF) \mbox{ has word }2314 \mbox{ (II).}\\
(1,2)\PF & \mbox{if } \word(\PF)=1324 \mbox{ and placement is (v). \ } \Phi(\PF) \mbox{ has word }2314 \mbox{ (III).}\\
(1,2)\PF & \mbox{if } \word(\PF)=1342 \mbox{ and placement is (iii). } \Phi(\PF) \mbox{ has word }2341 \mbox{ (I).}\\
(1,2)\PF & \mbox{if } \word(\PF)=1342 \mbox{ and placement is (v). \ } \Phi(\PF) \mbox{ has word }2341 \mbox{ (III).}
\end{cases}
\end{equation}
The first case above gives the fixed points of $\Phi$.

It is easy to check that the map $\Phi$ preserves area and the dinv since $\Phi$ does not change the Dyck path of $\pi$, and it also preserves the cars other than $\{1,2,3,4\}$. Since $\Phi$ changes the sign of the weight of each non-fixed point, it follows immediately that $\Phi$ forms a sign-reversing involution of the set of parking functions in $\PFc_{3,n}$ with pideses of either $\lambda 13 \mu$ or $\lambda 22 \mu$. As we mentioned, the set of fixed points of this involution is 
$$\mathrm{fp}(\Phi) = \{\PF\in\PFbb\ :\ \word(\PF)=3412, \mbox{ or } \word(\PF)=3142\mbox{ (i), (iii), (iv)}\}.$$

If we apply the involution $\Phi$ to all the parking functions $\PF\in\PFc_{3,n}$ that
we compute $\pides(\PF)$ and scan from left to right to find the first occurrence of either $(1,3)$ or non-fixed $(2,2)$ and apply $\Phi$ at that position. Then clearly, we have
\begin{theorem}\label{theorem:involution}
$\Phi$ is a sign-reversing involution of the set of parking functions in $\PFc_{3,n}$.
\end{theorem}
The fixed points in $\PFc_{3,n}$ have weakly decreasing pides in the form $3^a2^b1^c$, and these parking functions contribute to the coefficients of the Schur functions. Thus, we have the Schur positivity of the $m=3$ case:
\begin{corollary}\label{corollary:3ncase}
The polynomial $\Qmn{3,n}(1)$ is Schur positive, and
\begin{equation}\label{PF2a1b}
\scoeff{2^a1^b}_{3,n}=\sum_{
	\substack{\PF\in\PFc_{3,n},\ \pides(\PF)=2^a1^b,\\
		\PF\textnormal{ fixed by }\Phi}
}
[\ret(\PF)]_{\frac{1}{t}} t^{\area(\PF)}q^{\dinv(\PF)}.
\end{equation}
\end{corollary}

\subsubsection{The map $\swi$ and the symmetry $[s_{2^a1^b}]_{3,n}=[s_{2^b1^a}]_{3,3(a+b)-n}$}
For any parking function $\PF\in\PFc_{3,n}$ with $\pides(\PF)=2^a1^b$, we study the placement of the $a$ pairs of numbers $$\{(1,2),(3,4),\ldots,(2a-1,2a)\}$$ and the $b$ singletons $$\{2a+1,\ldots,2a+b\}.$$ Note that the two cars in each pair cannot be placed in the same column since the rank of the smaller car is larger than the rank of the bigger car.

Since there are 3 columns, we have $\binom{3}{2}$ ways to choose columns for each pair $(2i-1,2i)$. We name the 3 columns from left to right by $\ell,c,r$. Once we determine 2 columns for the pair, the filling of the two cars in the pair is fixed by their ranks since $\rank(2i-1)>\rank(2i)$. Now, we define the notation for the placement of a pair $(2i-1,2i)$:
\begin{enumerate}
	\item $L$ means $(2i-1,2i)$ are in the left 2 columns $\ell,c$,
	\item $R$ means $(2i-1,2i)$ are in the right 2 columns $c,r$,
	\item $C$ means $(2i-1,2i)$ are in columns $\ell,r$.
\end{enumerate} 
Similarly, we have $\binom{3}{1}$ ways to choose a column for each singleton. For a singleton $j$, we define the notation for the placement:
\begin{enumerate}
	\item $L$ means $j$ is in the left column $\ell$,
	\item $R$ means $j$ is in the right column $r$,
	\item $C$ means $j$ is in column $c$.
\end{enumerate}

Now we are ready to describe the map $\swi$.
Given a parking function $\PF\in\PFc_{3,n}$ with $\pides(\PF)=2^a1^b$, we track the placements of the $a$ pairs of cars $\{(1,2),\ldots,(2a-1,2a)\}$ and the $b$ singleton cars $\{2a+1,\ldots,2a+b\}$. Let the $a+b$ placements of these $a+b$ objects be $p_1,\ldots,p_a,p_{a+1},\ldots,p_{a+b}$ (here $p_i$ is one of $L$, $R$ or $C$). 

Then we consider $b$ pairs of cars $\{(1,2),\ldots,(2b-1,2b)\}$ and $a$ singleton cars $\{2b+1,\ldots,2b+a\}$. We build a new parking function $\swi(\PF)$ by first counting how many cars in each column if we assign the $b+a$ placements $p_{b+a},\dots,p_1$ to the $b+a$ objects in this reversed order, then constructing the path according to the number of cars of each column. Finally, we fill from  the first pair $(1,2)$ to the last singleton $2b+a$ based on the rule that $\rank(2i-1)<\rank(2i)$ for $i\leq b$ for the $b$ pairs of cars and the column placement choice $p_{b+a},\dots,p_1$. We call this map the switch map $\mathbb{S}$. \fref{abba} shows an example that we can construct a parking function in $\PFc_{3,5}$ with $\pides\ 21^3$ from a parking function in $\PFc_{3,7}$ with $\pides\ 2^31$.
\begin{figure}[ht!]
	\centering
	\resizebox{\columnwidth}{!}{\begin{tikzpicture}[scale=0.4]
		\fillshade{1/1,1/2,1/3,2/3,2/4,2/5,3/5,3/6,3/7};
		\PFmnu{0,0}{3}{7}{0/2,0/4,0/7,1/1,1/3,1/5,2/6,0/0};
		\arrowx{3.2,2.5}{4.8,2.5};
		\end{tikzpicture}
		\begin{tikzpicture}[scale=0.4]
		\draw[help lines] (0,0) grid (3,4);
		\fillsome{1/1/1,2/1/2,1/2/3,2/2/4,2/3/5,3/3/6,1/4/7,0/1/L,0/2/L,0/3/R,0/4/L};
		\arrowx{3.2,2.5}{4.8,2.5};
		\end{tikzpicture}
		\begin{tikzpicture}[scale=0.4]
		\draw[help lines] (0,0) grid (3,4);
		\fillsome{1/1/1,2/1/2,3/2/3,1/3/4,1/4/5,0/1/L,0/2/R,0/3/L,0/4/L};
		\arrowx{3.2,2.5}{4.8,2.5};
		\end{tikzpicture}
		\begin{tikzpicture}[scale=0.4]
		\draw[help lines] (0,0) grid (3,5);
		\draw[ultra thick,blue] (0,0) -- (0,3) -- (1,3) -- (1,4) -- (2,4) -- (2,5) -- (3,5);
		\fillsome{1/1/2,2/4/1};
		\arrowx{3.2,2.5}{4.8,2.5};
		\end{tikzpicture}
		\begin{tikzpicture}[scale=0.4]
		\draw[help lines] (0,0) grid (3,5);
		\draw[ultra thick,blue] (0,0) -- (0,3) -- (1,3) -- (1,4) -- (2,4) -- (2,5) -- (3,5);
		\fillsome{1/1/2,2/4/1,3/5/3};
		\arrowx{3.2,2.5}{4.8,2.5};
		\end{tikzpicture}
		\begin{tikzpicture}[scale=0.4]
		\draw[help lines] (0,0) grid (3,5);
		\draw[ultra thick,blue] (0,0) -- (0,3) -- (1,3) -- (1,4) -- (2,4) -- (2,5) -- (3,5);
		\fillsome{1/1/2,2/4/1,3/5/3,1/2/4};
		\arrowx{3.2,2.5}{4.8,2.5};
		\end{tikzpicture}
		\begin{tikzpicture}[scale=0.4]
		\fillshade{1/1,1/2,2/2,2/3,2/4,3/4,3/5};
		\PFmnu{0,0}{3}{5}{0/2,0/4,0/5,1/1,2/3,0/0};
		\end{tikzpicture}}
	\caption{$\PF\in\PFc_{3,7}$ with $\pides\ 2^31$ and $\swi(\PF)\in\PFc_{3,5}$ with $\pides\ 21^3$.}
	\label{fig:abba}	
\end{figure}

We shall  prove several properties of the map $\mathbb{S}$. It is even not obvious that the image of a parking function is still above the diagonal, thus we shall show that
\begin{theorem}\label{theorem:feas}
	If $\PF$ is a $(3,n)$-parking function with pides $2^a1^b$, then $\swi(\PF)$ is also a parking function.
\end{theorem}
\begin{proof}
	We still consider the $a$ pairs of cars $\{(1,2),\ldots,(2a-1,2a)\}$ and the $b$ singleton cars $\{2a+1,\ldots,2a+b\}$ of $\pi$.
	Suppose that there are $\ell_1,c_1,r_1$ placements of the first $a$ pairs of cars which are $L,\ R$ and $C$ respectively, and $\ell_2,c_2,r_2$ placements of the last $b$ singleton cars which are $L,\ R$ and $C$ respectively. Without loss of generality, we suppose that $n=3k+1$. Then we have 
	\begin{align}
	\ell_1+c_1+r_1=&a,\\
	\ell_2+c_2+r_2=&b,\\
	2a+b=&3k+1.
	\end{align}
	
	Since $\PF$ is a parking function, the path of the parking function should be above the diagonal, thus the number of cars in the left column is at least $k+1$ and the number of cars in the left two columns is at least $2k+1$. 
	
	Note that an $L$ placement of a pair contribute 1 left car and 1 center car, a $C$ placement of a pair contribute 1 left car and 1 right car, and an $R$ placement of a pair contribute 1 right car and 1 center car. The contribution of the singleton cars are obvious. Thus the number of cars in the left column is $\ell_1+c_1+\ell_2$, and the number of cars in the left 2 columns is $2\ell_1+r_1+c_1+\ell_2+c_2=a+\ell_1+\ell_2+c_2$, and we have that
	\begin{align}
	\label{ine1}\ell_1+c_1+\ell_2 \geq& k+1,\\
	\label{ine2}a+\ell_1+\ell_2+c_2\geq& 2k+1.
	\end{align} 
	
	Next, for $\swi(\PF)$, it has $\ell_2,c_2,r_2$ placements of the first $b$ pairs of cars which are $L,\ R$ and $C$ respectively, and $\ell_1,c_1,r_1$ placements of the last $a$ singleton cars which are $L,\ R$ and $C$ respectively. The total number of cars is equal to $2a+b=3(a+b)-(2b+a)=3(a+b)-3k-1=3(a+b-k-1)+2$, the number of cars in the left column should be at least $a+b-k=(2a+b)-a-k=3k+1-a-k=2k+1-a$, and the number of cars in the left two columns should be at least $2a+2b-2k=b+(3k+1)-2k=b+k+1$.  $\swi(\PF)$ is a parking function if the following equations hold:
	\begin{align}
	\label{ine3}\ell_1+c_2+\ell_2 \geq& 2k+1-a,\\
	\label{ine4}b+\ell_1+\ell_2+c_1\geq& b+k+1.
	\end{align} 
	Clearly, (\ref{ine1}) implies (\ref{ine4}), (\ref{ine2}) implies (\ref{ine3}).
\end{proof}

Next, we have the formula for area.
\begin{theorem}\label{theorem:areasw}
	Let $\PF$ be a $(3,n)$-parking function with pides $2^a1^b$. Using the definitions of \\$\ell_1,c_1,r_1,\ell_2,c_2,r_2$ in the proof of \tref{feas}. Let $L=\ell_1+\ell_2,R=r_1+r_2,C=c_1+c_2$, then 
	\begin{equation}
	\area(\PF)=L-R-1.
	\end{equation}
\end{theorem}
\begin{proof}
	We want to compute the area of a parking function as the difference of its coarea and the maximum coarea of a $(3,2a+b)$-parking function.
	The maximum coarea of a $(3,2a+b)$-parking function is $\frac{(2a+b-1)(3-1)}{2}=2a+b-1$.
	
	Notice that the cars in the right column contribute 2 to coarea, and the cars in the center column contribute 1 to coarea, thus the coarea of $\pi$ is
	\begin{equation}
	\ell_1+3r_1+2c_1+2r_2+c_2=a+2(r_1+r_2)+(c_1+c_2)=a+2R+C.
	\end{equation}
	Then,
	\begin{equation}
	\area(\PF)=2a+b-1-(a+2R+C)=a+(L+R+C)-1-(a+2R+C)=L-R-1.\qedhere
	\end{equation}
\end{proof}
It follows immediately from \tref{areasw} that 
\begin{theorem}\label{theorem:areasw1}For any $\PF\in\PFc_{3,n}$ with $\pides(\PF)=2^a1^b$,
	\begin{equation}
	\area(\PF)=\area(\swi(\PF)).
	\end{equation}
\end{theorem}

We have not yet proved, but verified all parking functions with less than or equal to 10 rows  for the following conjecture:
\begin{conjecture}\label{conjecture:2a1bn}
	For any $\PF\in\PFc_{3,n}$ with $\pides(\PF)=2^a1^b$,
	\begin{enumerate}[(a)]
		\item $\dinv(\PF)=\dinv(\swi(\PF)).$
		\item When $n$ and 3 are coprime,
		if $\PF$ is a fixed point of $\Phi$, then so is $\swi(\PF)$, and $\pides(\swi(\PF))=2^b1^a$.
	\end{enumerate}
\end{conjecture}

Notice that $\swi$ does not preserve the ``return" statistic, and the return of any parking function in the coprime case is 1 which is trivial. In the non-coprime case when $n$ is a multiple of 3, \cref{2a1bn} (b) will fail.
Further, we have the following results:
\begin{theorem}\label{theorem:imply}
In the case when 3 and $n$ are coprime, \cref{2a1bn} (b) implies that the map $\swi$ is a bijection between the fixed parking functions with $\pides\ 2^a1^b$ and the fixed parking functions with $\pides\ 2^b1^a$, 
and
\begin{equation}\label{2a1br}
[s_{2^a1^b}]_{3,n}=[s_{2^b1^a}]_{3,3(a+b)-n}.
\end{equation}
\end{theorem}
\begin{proof}
The bijectivity follows from the fact that the map $\swi$ is invertible. Equation (\ref{2a1br}) follows from the bijectivity and Equation (\ref{PF2a1b}).
\end{proof}

\subsubsection{The switch map $\swi$ in the $m$ column case}
We haven't completely understood how to use straightening to compute 
the coefficients of $s_{\lambda}$ for general $(m,n)$ case, but 
computations in Maple have led us to conjecture the following:

\begin{conjecture}\label{mcolswitch} For all $m,n > 0$ and $\alpha_i \geq 0$,
	\begin{equation}
	[s_{(m-1)^{\alpha_{m-1}}(m-2)^{\alpha_{m-2}}\cdots1^{\alpha_{1}}}]_{m,n}=[s_{(m-1)^{\alpha_{1}}(m-2)^{\alpha_{2}}\cdots1^{\alpha_{m-1}}}]_{m,(m\sum_{i=1}^{m-1}\alpha_i-n)}.
	\end{equation}
\end{conjecture}

On the other hand, the switch map $\swi$ that we have defined for the three column case can be naturally generalized to the $m$ column case, which conjecturally has many nice properties and is considered to be useful in proving Conjecture \ref{mcolswitch}. The definition of an $m$ columns switch map will need some new definitions.

Given any parking function $\PF\in\PFc_{m,n}$, we suppose that $s,s+1,\ldots,s+r-1$ form an increasing subsequence of the word $\sg(\PF)$, then by Remark \ref{remark:1}, the cars $s,s+1,\ldots,s+r-1$ must be placed in r different columns in a rank decreasing way.

There are $\binom{m}{r}$ possible choices to pick r columns for such cars $s,s+1,\ldots,s+r-1$. Let $p=\{s_1,s_2,\ldots,s_r\}\subset \{1,\ldots,m\}$ be a possible placement (i.e.\ the choice of columns), then we define the \emph{reverse complement} of $p$ to be $p^{rc}:=\{1,\ldots,m\}\backslash\{m+1-s_r,\ldots,m+1-s_1\}$, which is a placement for $m-r$ cars.

Given $\mu=\mu_1\cdots\mu_k \models n$. Suppose that 
$\sg(\PF)$, the word of $\PF$, is a shuffle of the increasing sequences $(1,\ldots,\mu_1),(\mu_1+1,\ldots,\mu_1+\mu_2),\ldots,(n-\mu_k+1,\ldots,n)$, and the placement  of the sequence $(\mu_1+\ldots+\mu_{i-1}+1,\ldots,\mu_1+\ldots+\mu_{i})$ is $p_i$, then we construct $\swi(\PF)$ as follows.

Let $\mu^{\vee}:=\mu^{\vee}_1\cdots\mu^{\vee}_k$, where $\mu^{\vee}_i:=m-\mu_{k+1-i}$.
We build the word of $\swi(\PF)$ to be a shuffle of $(1,\ldots,\mu^{\vee}_1),(\mu^{\vee}_1+1,\ldots,\mu^{\vee}_1+\mu^{\vee}_2),\ldots,(n-\mu^{\vee}_k+1,\ldots,n)$, and choose $p_{k+1-i}^{rc}$ as the placement of $(\mu^{\vee}_1+\ldots+\mu^{\vee}_{i-1}+1,\ldots,\mu^{\vee}_1+\ldots+\mu^{\vee}_{i})$. This construction is well defined, and we can invert the the map when the composition $\mu$ is given.

For example, suppose that there are $m=4$ columns. Take $\mu=(2,2,3)\vDash n$ where $n=7$. For a $(4,7)$-parking function $\PF$ whose word $\sg(\PF) = 5613472$ is a shuffle of $(1,2),(3,4),(5,6,7)$ with placements $\{1,3\},\{1,2\},\{1,2,4\}$, we construct $\swi(\PF)$ such that its word is a shuffle of $(1),(2,3),(4,5)$ and the placements are $\{2\},\{1,2\},\{1,3\}$, shown in \fref{exswi}.

\begin{figure}[ht!]
	\centering
	{\begin{tikzpicture}[scale=0.5]
		\fillshade{1/1,1/2,2/2,2/3,2/4,3/4,3/5,3/6,4/6,4/7};
		\PFmnu{0,0}{4}{7}{0/2,0/4,0/6,1/3,1/5,2/1,3/7,0/0};
		\arrowx{4.3,1.5}{5.9,1.5};
		\end{tikzpicture}
		\begin{tikzpicture}[scale=0.5]
		\draw[help lines] (0,0) grid (4,3);
		\fillsome{1/1/1,3/1/2,1/2/3,2/2/4,1/3/5,2/3/6,4/3/7,{-1}/1/{\{1,3\}},{-1}/2/{\{1,2\}},{-1}/3/{\{1,2,4\}}};
		\arrowx{4.3,1.5}{5.9,1.5};
		\end{tikzpicture}
		\begin{tikzpicture}[scale=0.5]
		\draw[help lines] (0,0) grid (4,3);
		\fillsome{2/1/1,1/2/2,2/2/3,1/3/4,3/3/5,
			{-.6}/3/{\{1,3\}},{-.6}/2/{\{1,2\}},{-.6}/1/{\{2\}}};
		\arrowx{4.3,1.5}{5.9,1.5};
		\end{tikzpicture}
		\begin{tikzpicture}[scale=0.5]
		\fillshade{1/1,1/2,2/2,2/3,3/3,3/4,4/4,4/5};
		\PFmnu{0,0}{4}{5}{0/3,0/4,1/1,1/2,2/5,0/0};
		\end{tikzpicture}}
	\caption{An example of $\pi$ and $\swi(\PF)$.}
	\label{fig:exswi}	
\end{figure}
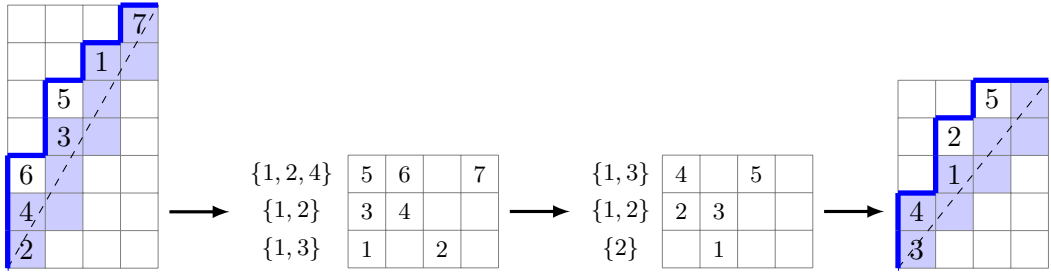

Using the same technique as the $3$ column case, we have
\begin{theorem}\label{switchmn1}
	$\PF$ is an $(m,n)$-parking function if and only if \ $\swi(\PF)$ is a parking function. Further, 
	$$
	\area(\PF)=\area(\swi(\PF)).
	$$
\end{theorem}

For a composition $\mu=\mu_1\cdots\mu_k \models n$, we say a permutation $\sg$ is a shuffle of $\mu$ if $\sg$ is a shuffle  of the increasing sequences $(1,\ldots,\mu_1),(\mu_1+1,\ldots,\mu_1+\mu_2),\ldots,(n-\mu_k+1,\ldots,n)$. Then we have:
\begin{theorem}\label{switchmn2}
	The switch map $\swi$ is a bijection between $(m,n)$-parking functions whose words are shuffle of $\mu=\mu_1\cdots\mu_k$ and $(m,mk-n)$-parking functions whose words are shuffle of $\mu^{\vee}=(m-\mu_{k})\cdots (m-\mu_1)$.
\end{theorem}

The switch map of $m$ column case still preserves the dinv statistic experimentally (summarized in the following conjecture), which we are not able to prove.
\begin{conjecture}\label{switchmn}For any $\PF\in\PFc_{m,n}$ where $\sg(\PF)$ is a shuffle of $\mu\models n$, 
	$$
	\dinv(\PF)=\dinv(\swi(\PF)).
	$$
\end{conjecture}

Thus conjecturally, the switch map $\swi$ is an \emph{area,dinv-preserving bijective map} between $(m,n)$-parking functions whose words are shuffle of $\mu$ and $(m,mk-n)$-parking functions whose words are shuffle of $\mu^{\vee}$. In the end, we shall discuss some results related to Theorem \ref{switchmn1}, Theorem \ref{switchmn2} and Conjecture \ref{switchmn}.

Referring to Haglund's work in \cite{Hagbook}, for any parking function whose word is a shuffle of $\mu=\mu_1\cdots\mu_k\models n$, we can replace the cars $\mu_1+\ldots+\mu_{i-1}+1,\ldots,\mu_1+\ldots+\mu_{i}$ with number $i$ to obtain a word parking function with cars $1^{\mu_1}\cdots k^{\mu_k}$ with the same area and dinv statistics. Further when $m$ and $n$ are coprime, we have 
\begin{equation}
\Qmn{m,n}(1)\bigg|_{m_\mu}=
\sum_{\PF\in\PFc_{m,n},\ \sg(\PF)\textnormal{ is a shuffle of }\mu}t^{\area(\PF)}q^{\dinv(\PF)}.
\end{equation} 
By the definition of  Hall scalar product, for any symmetric function $f$, we have
\begin{equation}
\langle f, h_{\mu}\rangle=f\big|_{m_\mu}.
\end{equation}
Thus, the properties of the switch map $\swi$ (Theorem \ref{switchmn1}, Theorem \ref{switchmn2} and Conjecture \ref{switchmn}) imply the coprime case of the following conjecture:
\begin{conjecture}\label{switchmnconj}
	For $m,n>0$, $\mu=\mu_1\cdots\mu_k\vdash n$ and $\mu^{\vee}=(m-\mu_{k})\cdots (m-\mu_1)$, 
	\begin{equation}
	\langle \Qmn{m,n}(1),h_\mu\rangle = \langle \Qmn{m,mk-n}(1),h_{\mu^{\vee}}\rangle.
	\end{equation}
\end{conjecture}

Since the area-preserving property of $\swi$ is proved in Theorem \ref{switchmn1}, we have the following theorem which is a special case of Conjecture \ref{switchmnconj}:
\begin{theorem}
	For coprime integers $(m,n)$, $\mu=\mu_1\cdots\mu_k\vdash n$ and $\mu^{\vee}=(m-\mu_{k})\cdots (m-\mu_1)$, 
	\begin{equation}
	\langle \Qmn{m,n}(1) \big|_{q=1},h_\mu\rangle =\langle \Qmn{m,mk-n}(1) \big|_{q=1},h_{\mu^{\vee}}\rangle.
	\end{equation}
\end{theorem}

\end{document}